\documentclass[preprint,12pt]{elsarticle}

\usepackage{amssymb}
\usepackage{amsmath}
\usepackage{amsthm}
\usepackage{epstopdf}
\usepackage{graphicx}
\usepackage{multirow}
\usepackage{times}
\usepackage{color}
\usepackage{makecell}
\usepackage{array}
\usepackage{booktabs}
\usepackage{threeparttable}
\usepackage{mathrsfs}
\usepackage{mathtools}

\theoremstyle{plain}
\newtheorem{theorem}{Theorem}[section]

\newtheorem{corollary}[theorem]{Corollary}
\newtheorem{proposition}[theorem]{Proposition}

\theoremstyle{definition}
\newtheorem{definition}[theorem]{Definition}

\theoremstyle{remark}
\newtheorem{remark}{Remark}

\journal{Omega}

\begin{document}

\begin{frontmatter}



\title{Preference Disaggregation Analysis with Criteria Selection in a Regularization Framework} 


\author[label1] {Kun Zhou}
\author[label1] {Zaiwu Gong \corref{fn1}}
\author[label2] {Guo Wei}
\author[label3,label4] {Roman S{\l}owi\'{n}ski}
\address[label1]{School of Management Science and Engineering, Nanjing University of Information Science and Technology, Nanjing 210044, China}
\address[label2]{The University of North Carolina at Pembroke, Pembroke, North Carolina 28372, USA}
\address[label3]{Institute of Computing Science, Pozna\'{n} University of Technology, Pozna\'{n} 60-965, Poland}
\address[label4]{Systems Research Institute, Polish Academy of Sciences, Warsaw 01-447, Poland}
\cortext[fn1]{Corresponding author \indent Email: zwgong26@163.com.}

\begin{abstract}
Limited by cognitive abilities, decision-makers (DMs) may struggle to evaluate decision alternatives based on all criteria in multiple criteria decision-making problems. This paper proposes an embedded criteria selection method derived from preference disaggregation technique and regularization theory. The method aims to infer the criteria and value functions used by the DM to evaluate decision alternatives. It measures the quality of criteria subsets by investigating both the empirical error (fitting ability of value functions to preference information) and generalization error (complexity of value functions). Unlike existing approaches that consider only the deviation from linearity as a measure of complexity, we argue that the number of marginal value functions also affects complexity. To address this, we use 0-1 variables to indicate whether a criterion is selected in the value function or not, and construct a criteria selection model with the trade-off between empirical and generalization errors as the objective function. If the criteria are sufficiently discriminative, we identify all supporting criteria sets that can restore preference information without unnecessary criteria. We further analyze the likelihood of criteria being selected by the DM. Finally, the effectiveness of the proposed method is demonstrated by applying it to an example of the green supplier selection problem.
\end{abstract}


\begin{keyword}
multiple criteria decision aiding; criteria selection; preference disaggregation; regularization; parsimonious preference representation; green supply chain management


\end{keyword}

\end{frontmatter}



\section{Introduction}
When evaluating decision alternatives (also called potential actions), real-world decision-making problems often require decision-makers (DMs) to assess the performance of alternatives from multiple points of view. Therefore, alternatives are typically evaluated based on a set of conflicting criteria. Multiple Criteria Decision Aiding (MCDA) provides numerous methods to assist DMs in handling such complex decision-making problems \citep{ref43,ref55}. The core idea is to characterize the DM's preferences by a mathematical model and then derive recommendations consistent with the preference model. Typically, such preference models are constructed or learned from some preference information provided by DMs \citep{HS_2024,ref61}. Preference information can be categorized into direct information and indirect information (also called decision examples).  Direct preference information is the set of values of model parameters such as trade-off weights and indifference thresholds, etc. This information requires DMs to have a thorough understanding of the parameters involved in preference models and their own preference systems \citep{ref46}. This places high demands on DMs' cognitive abilities and background knowledge. Indirect preference information is the holistic judgment of alternatives in terms of pairwise comparisons \citep{ref2,Zhou} or assignments to decision classes \citep{ref46,ref1}. This information requires less cognitive effort from DMs and avoids the need for DMs to directly specify preference model parameters and explain the reasons behind their values \citep{ref56}. 

Preference disaggregation technique \citep{ref38,ref12,ref3}, widely recognized in MCDA community, comprises a general framework for inferring DMs' preference models on the basis of decision examples. This technique aims to find a rational basis that supports the DM to give such decision examples. Preference disaggregation technique establishes a connection between decision examples and mathematical preference models by regression-like methods. The preference model that best fits the decision examples is selected to represent the DM's preference system. The preference models can take various forms, including value function \citep{ref2,ref13,ref14}, outranking relation \citep{buchong3,buchong9} or a set of decision rules \citep{buchong2}. Disaggregation technique shares the same philosophy as machine learning, both considering the problem of learning a (decision/prediction) model from input information \citep{buchong10}. The existing disaggregation methods are typically developed with the assumption that DMs evaluate alternatives using all involved criteria. However, cognitive theory suggests that humans exhibit selective attention when comparing things \citep{ref16}. Cognitive limitations result in an increased cognitive burden to consider all relevant points of view when evaluating alternatives \citep{buchong12}, so humans cannot focus on all details. Instead, they concentrate on the aspects they care about \citep{buchong12}. Therefore, if there are many criteria, it cannot be arbitrarily assumed that DMs evaluate alternatives using all criteria. Consequently, there is a need for an efficient and interpretable method to infer, from DMs' preference information, the criteria they select for evaluating alternatives as well as the preference model that aggregates these criteria.

In this paper, we adopt additive value function as the preference model and propose an intuitive and interpretable embedded criteria selection method based on preference disaggregation technique and regularization theory. The philosophy of this method is to determine a value function that is as ``simple'' as possible while maintaining its ability to fit decision examples. A ``simple'' value function can effectively avoid over-fitting, which refers to good fitting ability to reference alternatives but poor generalization ability to non-reference alternatives \citep{ref41}. We use the fitting ability (empirical error) and complexity (generalization error) of value functions as the quality measures of criteria subsets to guide criteria selection. Unlike existing literature which interprets the complexity of value functions solely as the degree of deviation of  marginal value functions from linearity \citep{ref1,ref6}, we argue that the complexity of value functions consists of two components: 1) the degree of deviation of marginal value functions from linearity and 2) the number of marginal value functions that compose the value function. In fact, reducing the number of criteria can also enhance the generalization ability of the learned model, which has been demonstrated in machine learning \citep{ref17}. We introduce 0-1 variables to indicate whether a criterion is selected or not, and the marginal value functions are multiplied by the corresponding 0-1 variables and then summed to form a new value function. The complexity of such a value function is jointly measured by the difference in slopes of marginal value functions in adjacent subintervals \citep{ref42} and the sum of all 0-1 variables.

The more criteria we select, the higher the degree of freedom of the value function constructed based on these criteria, and the higher the value function ability to fit decision examples. However, this increases the complexity of the value function. To address the issue of parsimonious preference representation, we balance the empirical error and generalization error of the value function to form a general regularized loss function. The decision examples are converted into constraints that the value function must satisfy, together with monotonicity and normalization conditions, to construct a criteria selection model. This is a nonlinear mixed-integer programming problem. To efficiently solve this problem, we transform it into an equivalent linear form. This model simultaneously performs criteria selection and learning of the value function, and the learned value function maintains a certain level of fitting ability to decision examples while having a simpler form.

If the involved criteria are sufficiently discriminative, there exists at least one subset of criteria enabling the construction of a value function being compatible with all decision examples. Such a subset of criteria is referred to as supporting criteria set for the preference information. It is evident that adding a new criterion to a supporting criteria subset still results in a supporting criteria set. However, such augmented supporting criteria set is not what we desire, as it contains unnecessary criteria. To determine all supporting criteria sets without unnecessary criteria (referred as streamlined supporting criteria set), we first construct a criteria selection model with the generalization error as the objective function to find the smallest streamlined supporting criteria set. Subsequently, an iterative method is developed to seek all possible streamlined supporting criteria sets. Based on all streamlined supporting criteria sets, we further define core criteria which are present in all streamlined supporting criteria sets, redundant criteria which are not included in any streamlined supporting criteria sets, and criteria relevance which is measured by the number of streamlined supporting criteria sets that include a particular criterion. They reveal the likelihood that these criteria are mainly selected by the DM.

The remainder of this article is organized as follows: In Section \ref{review}, an overview of the related work is presented. Section \ref{UTAreview} briefly introduces the UTA method, which is the best-known preference disaggregation method. In Section \ref{new value function}, we rewrite the additive value function by introducing 0-1 variables and study the complexity of the novel value function. At the same time, some important concepts and definitions are also introduced for ease of subsequent discussion. Criteria selection models for the two scenarios of consistent and inconsistent preference information are elaborated in Section \ref{criteria selection model}. In Section \ref{case study}, an example of green supplier selection problem is provided to validate the effectiveness of the proposed criteria selection method. Section \ref{conclusion} draws conclusions.

\section{Literature Review}\label{review}
\subsection{The disaggregation methods with a value function as the preference model}
Among various disaggregation methods, those using value functions as preference models are the most prevalent. As a typical representative, the UTA method \citep{ref2,ref4} uses an additive value function as the preference model and approximates the additive value function through a piecewise linear function to facilitate the inference of the preference model. By using a regression-like method, the decision examples are converted into constraints that the parameters of value functions need to satisfy, and an optimization process is used to select the value function with the highest degree of restoration to the decision examples as the preference model. However, the UTA method requires the value function to be monotonic and the criteria to be mutually independent, which may not be suitable for certain decision scenarios. Some methodologies are then proposed to address these issues. \citet{buchong4} proposed a method based on an evolutionary optimization approach to infer non-monotonic value functions. \citet{ref7} developed a new framework for preference disaggregation to infer non-monotonic value functions. This approach first constructs a linear programming model to determine a value function that is not normalized. Subsequently, a mapping that preserves the order relations is developed to transform this value function into a normalized one. Regarding the interactions among criteria, \citet{ref13} utilized the Choquet integral as a preference model to extend the UTA method. Apart from decision examples, the newly proposed method allows the DM to provide information about the interactions among criteria. \citet{buchong5} added positive and negative synergy terms to the additive value function to depict the positive and negative interactions of criteria respectively, proposing a novel disaggregation method. This method does not require the evaluation of criteria to be expressed on the same scale. Furthermore, \citet{ref14} introduced a compensatory value function by dividing the trade-off mechanism between criteria into independent and dependent parts. Such value function can simultaneously handle the importance of criteria, interactions between criteria and the DM's risk tolerance. 

A common principle shared by the aforementioned methods is to emphasize the fitting ability of the value function to decision examples, which may lead to a complex value function. However, the generalization ability of complex value functions on non-reference alternatives is poor. Therefore, we need to strike a balance between the fitting ability of the value function to the decision examples and its complexity. Regularization theory is an important methodology in machine learning, used to address the trade-off between model complexity and generalization performance \citep{ref21}. Its core operation is to construct a penalty term (also called regularization term) related to the model complexity based on the model parameters, and incorporate the constructed penalty term as a part of the objective function. \citet{ref41} extended the UTADIS method \citep{buchong6} by utilizing the $L_{1}$ norm of the parameters of the value function as a penalty term for model complexity. In the extension, both the empirical error and complexity of the value function are considered, and the performance of such extension UTADIS method has been proven superior to UTADIS through computational experiments. \citet{buchong7} introduced a regularization term based on the $L_{2}$ norm to extend the additive value function and discusses how the optimization process can be expressed in a kernel form. \citet{ref6} argued that the closer the value function is to linearity, the simpler it is, and they regard the difference in slopes of marginal value functions in adjacent subintervals as a penalty term for model complexity.

In this paper, the complexity of the value function is divided into two parts: one is the degree to which the marginal value function deviates from linearity, and the other is the number of marginal value functions that constitute the value function. The regularization term constructed in this way not only controls the degree to which the value function deviates from linearity, but also possesses the function of criteria selection. 

\subsection{Existing methods for dealing with irrelevant criteria}
The problems of irrelevant criteria (also called features or attributes) have been extensively studied in machine learning \citep{ref17,ref18} (known as feature selection) and rough set theory \citep{ref19,ref24,ref59,ref20}) (known as attribute reduction). In machine learning, preprocessing data by feature selection methods can reduce data dimentions mitigating the curse of dimensionality \citep{ref21}. Concerning different selection strategies, feature selection methods can be broadly categorized as filter, wrapper, and embedded methods. Filter methods are independent of subsequent learning algorithms, and they assess the relevance of features based on the inherent characteristics of data \citep{ref27,ref22}. A typical filter method consists of two steps: 1) ranking features according to some feature evaluation criteria; and 2) filtering out low-ranked features. Filter methods possess high computational efficiency, but the absence of guidance from learning algorithms in feature selection results in poorer performance of learned models. Wrapper methods rely on the predictive performance of learning algorithms to evaluate the quality of selected features and consist of two operations: 1) selecting a feature subset according to certain strategy; and 2) using a learning algorithm as a black box to evaluate the quality of the selected feature subset based on learning performance. These operations are repeated until the optimal feature subset is found or a stopping condition is met \citep{ref28,ref31}. The feature subset determined by the wrapper methods gives the highest learning performance, but this comes at the cost of exponential computational and storage requirement (the search space for $d$ features is $2^{d}$). The embedded method is a tradeoff between filter and wrapper methods by embedding feature selection into the learning algorithm, simultaneously selecting features and learning model  \citep{ref33}. While retaining interaction with learning algorithm, this method avoids enumerating all feature subsets like wrapper methods. Representative embedded methods are based on sparse regularization models \citep{ref35,ref36}. These three feature selection methods are not suitable for MCDA due to three reasons  \citep{ref38}: 1) these methods are all model-oriented, focusing primarily on the performance of the developed models, and this often comes at the cost of complexity of both methodology and model, making them difficult for DMs to understand; 2) machine learning and MCDA have their respective application areas, where the former primarily deals with data-intensive problems, while the latter confronts small-scale data, emphasizing the active participation of the DM. This results in the above three feature selection methods being unable to be directly applied to decision problems; and 3) these methods lack interaction with DMs and only use initial statistical samples as an input, however, in decision-making problems, DMs' preference information may change in the course of interaction as DMs gain a deeper understanding of the feedback between the provided preference information and the resulting recommendation (for example, DMs may add, change or delete their preference information if this information is found to be inconsistent or the recommendation is not satisfactory).

Attribute reduction is the main task in rough set theory, aiming to remove redundant attributes while maintaining the classification ability of the information table \citep{ref58,ref23}. This method is essentially a data preprocessing process, which completes attribute reduction based on the inherent characteristics of the input information (decision table). The key step is to check which attributes are necessary to represent classification without increased inconsistency. It is known that determining all the reducts is NP-hard \citep{ref24}. Therefore, genetic algorithms are often used for attribute reduction. According to the adopted measure of attribute relevance, attribute reduction methods can be roughly divided into three categories: 1) attribute reduction methods based on positive region \citep{ref26,ref29}; 2) attribute reduction methods based on discernibility matrix \citep{ref30,ref37}; and 3) attribute reduction methods utilizing entropy \citep{ref40,ref39}. There are two limitations in attribute reduction methods that restrict its application in MCDA: 1) attribute reduction is a data preprocessing process that, similar to filter methods, is independent of model learning algorithms; 2) attribute reduction methods lack interaction with DMs, taking the initial given information as input only, and usually not allowing DMs to change information. These limitations have been overcome, however, in case of MCDA using as preference model a set of ``\textit{if..., then...'}'decision rules presenting relationships between values on relevant attributes and the decision \citep{ref58}. In this case, the attribute reduction is performed together with rule generation, taking for the rules the minimal subsets of relevant attributes for a given decision \citep{ref60}. 

\section{A brief reminder of the UTA method}\label{UTAreview}
UTA (UTilit\'{e}s Additives) is a preference disaggregation method using an additive value function as the preference model \citep{ref2}. In the framework of UTA, it is assumed that the performance scores (ordinal or cardinal scales) of alternatives on a finite set of criteria are known, and the DM is able to provide decision examples concerning a subset of alternatives called reference alternatives.

Assuming that alternatives $a_{i}, \ i\in N=\{1,2,\cdots,n\}$ are evaluated on criteria $c_{j},\ j\in M=\{1,2,\cdots,m\}$, let $\mathcal{A}=\{a_{i}\mid i\in N\}$ and $\mathcal{C}=\{c_{j}\mid j\in M\}$. The function $g_{j}(\cdot):\mathcal{A}\rightarrow[\underline{g_{j}},\overline{g_{j}}]$ gives performance scores to alternatives on criterion $c_{j}$, with $[\underline{g_{j}},\overline{g_{j}}]$ being the evaluation scale. Typically, $\underline{g_{j}}=\min\limits_{a_{i}\in \mathcal{A}}g_{j}(a_{i})$,  $\ \overline{g_{j}}=\max\limits_{a_{i}\in \mathcal{A}}g_{j}(a_{i})$. The UTA method uses the following additive value function to aggregate criteria scores:
\begin{equation}\label{value function}
  u(a_{i})=\sum\limits_{j=1}^{m}u_{j}(g_{j}(a_{i})),
\end{equation}
where $u_{j}(\cdot):[\underline{g_{j}},\overline{g_{j}}]\rightarrow[0,1]$ is the marginal value function of $c_{j}$, which is a piecewise linear function with characteristic points $g_{j}^{1},g_{j}^{2},\cdots,g_{j}^{\gamma_{j}-1}$ \citep{ref2}. These characteristic points divide $[\underline{g_{j}},\overline{g_{j}}]$ into $\gamma_{j}$ equal subintervals. Let $g_{j}^{0}=\underline{g_{j}},\ g_{j}^{\gamma_{j}}=\overline{g_{j}}$.

For $g_{j}(a_{i})\in [g_{j}^{t},g_{j}^{t+1}]$, we have
\begin{equation}\label{interpolation}
 \begin{split}
    u_{j}(g_{j}(a_{i}))&=u_{j}(g_{j}^{t})+\frac{g_{j}(a_{i})-g_{j}^{t}}{g_{j}^{t+1}-g_{j}^{t}}\times (u_{j}(g_{j}^{t+1})-u_{j}(g_{j}^{t}))\\
      &=u_{j}(g_{j}^{0})+\sum\limits_{s=1}^{t}\triangle u_{j}^{s}+\frac{g_{j}(a_{i})-g_{j}^{t}}{g_{j}^{t+1}-g_{j}^{t}}\times \triangle u_{j}^{t+1},
   \end{split}
\end{equation}
where $\triangle u_{j}^{s}=u_{j}(g_{j}^{s})-u_{j}(g_{j}^{s-1})$ is the value difference of $[g_{j}^{s-1},g_{j}^{s}]$. The value function $u(\cdot)$ needs to satisfy monotonicity and normalization conditions:
\begin{itemize}
  \item \textbf{Monotonicity conditions:} $\ \triangle u_{j}^{s}\geq\rho\geq0,\ s=1,2,\cdots,\lambda_{j},\ j\in M$.
  \item \textbf{Normalization conditions:} $\begin{cases}
          u_{j}(g_{j}^{0})=0,\ j\in M,\\
          \sum\limits_{j=1}^{m}\sum\limits_{s=1}^{\lambda_{j}}\triangle u_{j}^{s}=1,
         \end{cases}$
\end{itemize}
where the parameter $\rho$ is used to control the magnitude of value difference, and is typically set to $0$. It can be seen that $u_{j}(\cdot)$  is determined by $\triangle u_{j}^{s},\ s=1,2,\cdots,\lambda_{j}$, which constitute the characteristic vector of $u_{j}(\cdot)$:
\begin{equation}
  \textbf{V}_{j}=(\triangle u_{j}^{1},\triangle u_{j}^{2},\cdots,\triangle u_{j}^{\lambda_{j}})^{\textrm{T}}.
\end{equation}
$\textbf{I}^{\textrm{T}}\times\textbf{V}_{j}$ is the weight of $c_{j}$, where $\textbf{I}$ is a column vector with all elements equal to $1$.

As a representation of the DM's preference system, the following relationship exists between $u(\cdot)$ and the DM's holistic judgement (preference ``$\succ$'' and indifference ``$\sim$''):
\begin{equation}
\left\{
\begin{aligned}
  a_{i}\succ a_{j} &\Leftrightarrow u(a_{i})>u(a_{j}), \\
  a_{i}\sim a_{j} &\Leftrightarrow u(a_{i})=u(a_{j}).
\end{aligned}
\right.
\end{equation}

Let $\textbf{A}_{j}(a_{i})=(\delta_{j}^{1}(a_{i}),\cdots,\delta_{j}^{\lambda_{j}}(a_{i}))^{\textrm{T}}$, for $p=1,\cdots,\lambda_{j}$,
\begin{equation}
  \delta_{g_{j}}^{p}(a_{i})=\left\{\begin{array}{ll}
                   1, & if~\ g_{j}(a_{i})> g_{j}^{p}, \\
                   \frac{g_{j}(a_{i})-g_{j}^{p-1}}{g_{j}^{p}-g_{j}^{p-1}}, & if~\ g_{j}^{p-1}\leq g_{j}(a_{i})\leq g_{j}^{p},\\
                   0, & if~\ g_{j}(a_{i})<g_{j}^{p-1}.
                 \end{array}\right.
\end{equation}
The marginal value function $u_{j}(\cdot)$ can be expressed as the inner product of two vectors as follows \citep{ref44}:
\begin{equation}
  u_{j}(g_{j}(a_{i}))=\textbf{V}_{j}^{\textrm{T}}\times \textbf{A}_{j}(a_{i}).
\end{equation}
Let $\textbf{V}=(\textbf{V}_{1}^{\textrm{T}},\cdots,\textbf{V}_{m}^{\textrm{T}})^{\textrm{T}},\ \textbf{A}(a_{i})
=(\textbf{A}_{1}^{\textrm{T}}(a_{i}),\cdots,\textbf{A}_{m}^{\textrm{T}}(a_{i}))^{\textrm{T}}$. Then, $u(a_{i})$ can be expressed in the following form:
\begin{equation}
  u(a_{i})=\textbf{V}^{\textrm{T}}\times \textbf{A}(a_{i}).
\end{equation}
$\textbf{V}$ is the parameter vector of $u(\cdot)$ and has the following form:
\begin{equation}
\left(\overbrace{\triangle u_{1}^{1},\cdots,\triangle u_{1}^{\lambda_{1}}}^{c_{1}},\cdots,\overbrace{\triangle u_{j}^{1},\cdots,\triangle u_{j}^{\lambda_{j}}}^{c_{j}},\cdots,\overbrace{\triangle u_{m}^{1},\cdots,\triangle u_{m}^{\lambda_{m}}}^{c_{m}}\right)^{\textrm{T}}.
\end{equation}
$\textbf{A}(a_{i})$ is a constant vector.

To determine the value function that best fits the DM's preference information,  overestimation error $\sigma_{+}(a_{i})$ and underestimation error $\sigma_{-}(a_{i})$ were introduced by \citet{ref4} to modify the additive value function as follows:
\begin{equation}
  u^{\prime}(a_{i})=u(a_{i})+\sigma_{+}(a_{i})-\sigma_{-}(a_{i}).
\end{equation}
Then, the following linear program is constructed:
\begin{align}\label{UTA}
\begin{array}{ll}
 \text{minimize}   & F=\sum\limits_{a_{i}\in A^{R}}(\sigma_{+}(a_{i})+\sigma_{-}(a_{i}))\\
 \text{subject to} & \begin{cases}
 E^{C}:\begin{cases}
 1(b):\triangle u_{j}^{s}\geq\rho\geq0,\ s=1,2,\cdots,\lambda_{j},\ j\in M,\\
 2(b):\begin{cases}
 u_{j}(g_{j}^{0})=0,\ j\in M,\\
 \sum\limits_{j=1}^{m}\sum\limits_{s=1}^{\lambda_{j}}\triangle u_{j}^{s}=1,
 \end{cases}
 \end{cases}\\
 E^{A^{R}}:\begin{cases}
 u^{\prime}(a_{p})=u^{\prime}(a_{q}),~\mbox{for}~\ a_{p},a_{q}\in A^{R}\ ~\&~\ a_{p}\sim a_{q},\\
 u^{\prime}(a_{p})\geq u^{\prime}(a_{q})+\varepsilon,~\mbox{for}~\ a_{p},a_{q}\in A^{R}\ ~\&~\ a_{p}\succ a_{q},\\
 \sigma_{+}(a_{p}),\sigma_{+}(a_{q}),\sigma_{-}(a_{p}),\sigma_{-}(a_{q})\geq0,~\mbox{for}~\ a_{p},a_{q}\in A^{R},
 \end{cases}\\
 \varepsilon>0,
\end{cases}
\end{array}
\end{align}
where $\varepsilon$ is an auxiliary variable used to convert strict inequalities into weaker ones. The set $A^{R}$ consists of reference alternatives on which the DM provides preference information. The objective function $F=\sum\limits_{a_{i}\in A^{R}}(\sigma_{+}(a_{i})+\sigma_{-}(a_{i}))$ measures the fitting ability of the value function to the preference information. If the optimal value $F^{\ast}$ equals $0$,  there exists a value function that can restore all preference information (Such value function is called compatible). If $F^{\ast}>0$, there are inconsistencies in the preference information, and the optimal solution of model (\ref{UTA}) is the value function with the minimum empirical error. To identify the pieces of preference information that lead to inconsistency, an integer programming method proposed by \citet{ref5} can be adopted. In fact, as early as 1981, Zionts et al.
\citep{buchong1,buchong11} developed a discrete multi-criteria decision method with a similar philosophy to the UTA method, which was utilized for addressing energy planning issues.

\section{Rewriting the additive value function for criteria selection}\label{new value function}
In this section, we introduce 0-1 variables to reformulate the additive value function for criteria selection and explore the complexity of the novel value function. Additionally, for the convenience of subsequent discussion, some important concepts and definitions are also provided.

By introducing 0-1 variable $\delta_{j}$ to identify whether $c_{j}$ is selected or not, i.e., whether the DM evaluates alternatives considering $c_{j}$,  the value function $u(\cdot)$ is rewritten as follows:
\begin{equation}\label{rewrite value function}
\begin{split}
  u_{\ast}(a_{i}) & =\sum\limits_{j=1}^{m}\delta_{j}\times u_{j}(g_{j}(a_{i})) \\
    & =\sum\limits_{j=1}^{m}\delta_{j}\times \textbf{V}_{j}^{\textrm{T}}\times \textbf{A}_{j}(a_{i}).
\end{split}
\end{equation}
If $\delta_{j}=0$,  the performance of alternatives on $c_{j}$  does not generate any value, and the DM does not take $c_{j}$ into account when evaluating alternatives. If $\delta_{j}=1$, the marginal value function $u_{j}(\cdot)$ corresponding to $c_{j}$ is a non-zero component of the value function $u_{j}(\cdot)$, indicating that the DM will take into account the performance of the alternatives on $c_{j}$ when evaluating them. It can be seen that $u_{\ast}(\cdot)$ is jointly determined by $\delta_{j},\ j\in M,$ and $\textbf{V}_{j}^{\textrm{T}},\ j\in M$ . The former indicates whether $c_{j}$  is considered by the DM, while the latter characterizes the shape of $u_{j}(\cdot)$.

The value function $u_{\ast}(\cdot)$ needs to satisfy both monotonicity and normalization:
\begin{itemize}
  \item \textbf{Monotonicity conditions:} $\triangle u_{j}^{s}\geq0,\ s=1,2,\cdots,\lambda_{j},\ j\in M.$
  \item \textbf{Normalization conditions:} $\begin{cases}
          u_{j}(g_{j}^{0})=0,\ j\in M,\\
          \sum\limits_{j=1}^{m}\delta_{j}\times\textbf{V}_{j}^{\textrm{T}}\times\textbf{I}=1.
         \end{cases}$
\end{itemize}
The method proposed in this paper requires parameter $\rho$ to be set to $0$. Recently, some scholars have extended the UTA method to the case of non-monotonic value functions \citep{ref6,ref7,ref8}. However, to highlight the research content of this paper, we do not present our methodology in the context of non-monotonic value functions.

\begin{definition}\label{set-based}
The value function $u(\cdot)$ is said to be constructed over $\mathcal{C}_{t}\subseteq \mathcal{C}$ ($\mathcal{C}_{t}$-based) if it is composed of marginal value functions corresponding to the criteria in $\mathcal{C}_{t}$. In this case, $u(\cdot)$ has the following form:
\begin{equation}
  u(\cdot)=\sum\limits_{\{j|c_{j}\in \mathcal{C}_{t}\}}\textbf{V}_{j}^{\textrm{T}}\times \textbf{A}_{j}(\cdot).
\end{equation}
\end{definition}
For example, $u^{\star}(\cdot)=\textbf{V}_{1}^{\textrm{T}}\times \textbf{A}_{1}(\cdot)+\textbf{V}_{3}^{\textrm{T}}\times \textbf{A}_{3}(\cdot)+\textbf{V}_{4}^{\textrm{T}}\times \textbf{A}_{4}(\cdot)$ is $\{c_{1},c_{3},c_{4}\}$-based value function, and $u^{\star\star}(\cdot)=\textbf{V}_{2}^{\textrm{T}}\times \textbf{A}_{2}(\cdot)+\textbf{V}_{5}^{\textrm{T}}\times \textbf{A}_{5}(\cdot)+\textbf{V}_{6}^{\textrm{T}}\times \textbf{A}_{6}(\cdot)$ is $\{c_{2},c_{5},c_{6}\}$-based value function. Furthermore, $\{c_{1},c_{3},c_{4}\}$ is a proper subset of $\{c_{1},c_{2},c_{3},c_{4}\}$, and $u^{\star}(\cdot)$ can also be considered as $\{c_{1},c_{2},c_{3},c_{4}\}$-based.  In this case, $u^{\prime}(\cdot)=\textbf{V}_{1}^{\textrm{T}}\cdot \textbf{A}_{1}(\cdot)+\textbf{V}_{2}^{\textrm{T}}\times \textbf{A}_{2}(\cdot)+\textbf{V}_{3}^{\textrm{T}}\times \textbf{A}_{3}(\cdot)+\textbf{V}_{4}^{\textrm{T}}\times \textbf{A}_{4}(\cdot)$, where $\textbf{V}_{2}=\textbf{0}$ ($\textbf{0}$ a column vector with all elements being zero).
\begin{definition}\label{degenerate}
Given a value function $u(\cdot)=\sum\limits_{\{j|c_{j}\in \mathcal{C}_{t}\}}\textbf{V}_{j}^{\textrm{T}}\times \textbf{A}_{j}(\cdot)$. If there exists $c_{j}\in\mathcal{C}_{t}$ such that $\textbf{V}_{j}=\textbf{0}$, $u(\cdot)$ is said to be degenerate.
\end{definition}
In the above examples, $u^{\star}(\cdot)=\textbf{V}_{1}^{\textrm{T}}\cdot \textbf{A}_{1}(\cdot)+\textbf{V}_{2}^{\textrm{T}}\cdot \textbf{A}_{2}(\cdot)+\textbf{V}_{3}^{\textrm{T}}\cdot \textbf{A}_{3}(\cdot)+\textbf{V}_{4}^{\textrm{T}}\cdot \textbf{A}_{4}(\cdot)$  is degenerate. For the degenerate value function, the performance of alternatives of the criteria corresponding to $\textbf{V}_{j}=\textbf{0}$ is of no value to DMs.

\begin{definition}\label{supporting}
Given a criteria subset $\mathcal{C}_{k}\subseteq\mathcal{C}$, if there exists a $\mathcal{C}_{k}$-based value function that restores the preference information, $\mathcal{C}_{k}$ is referred to as the supporting criteria set for this preference information.
\end{definition}
The concept of supporting criteria set is extremely important. Criteria in the supporting criteria set are sufficiently discriminatory relative to preference information and the DM may evaluate alternatives considering only these criteria. Typically, there may be multiple supporting criteria sets.

\begin{proposition}\label{gg}
Given two criteria subsets, $\mathcal{C}_{k}$ and $\mathcal{C}_{t}$, satisfying $\mathcal{C}_{k}\neq\emptyset$ and $\mathcal{C}_{k}\subseteq\mathcal{C}_{t}$. If $\mathcal{C}_{k}$ is a supporting criteria set for preference information, $\mathcal{C}_{t}$ is also a supporting criteria set.
\end{proposition}
\begin{proof}
See \ref{A1}.
\end{proof}

\begin{corollary}\label{ff}
If $\mathcal{C}$ is not a supporting criteria set for the preference information, any criteria subset of $\mathcal{C}$ is also not a supporting criteria set. In this case, the preference information is partially inconsistent.
\end{corollary}
If $\mathcal{C}_{k}$ is a proper subset of $\mathcal{C}_{t}$ and both are supporting criteria sets of preference information, then the performances of alternatives on the criteria from $\mathcal{C}_{k}$ are sufficient to restore the preference information given by the DM, while the criteria from $\mathcal{C}_{t}\setminus\mathcal{C}_{k}$ are considered unnecessary.
\begin{definition}\label{reduction}
Assuming that $\mathcal{C}_{k}$ is a supporting criteria set, if no proper subset of $\mathcal{C}_{k}$ is supporting criteria set, $\mathcal{C}_{k}$ is referred to as streamlined supporting criteria set.
\end{definition}
The streamlined supporting criteria set does not contain unnecessary criteria. Following the Occam's Razor Law, we only focus on such supporting criteria set when selecting criteria.

If the DM's preference information is inconsistent, according to Corollary (\ref{ff}), it is impossible to construct a compatible value function based on any criteria subset to restore this preference information. In such cases, we focus on the fitting ability of the value function to the preference information. The more criteria included in $\mathcal{C}_{k}$, the higher the degree of freedom of $\mathcal{C}_{k}$-based value functions, and the better their fitting ability to preference information.

\begin{remark}\label{aa}
Assume that $\mathcal{C}_{k}$ is a non-empty and proper subset of $\mathcal{C}_{t}$. Then, there exists a $\mathcal{C}_{t}$-based value function whose empirical error is not greater than the empirical error of any $\mathcal{C}_{k}$-based value function.
\end{remark}

The validity of Propositions \ref{gg} and Remark \ref{aa} requires that $\rho=0$, i.e., the value difference of subintervals is allowed to be $0$. In this case, the optimal solution of model (\ref{UTA}) may determine a degenerate value function, where some marginal value function is $0$. While this indeed achieves the purpose of criteria selection, we cannot guarantee that this situation always occurs.

Next, we discuss the complexity of $u_{\ast}(\cdot)$. Using notation from statistical learning, $\Omega(u_{\ast})$ represents the complexity of $u_{\ast}$. In the literature, the deviation of $u_{\ast}(\cdot)$  from linearity is used as a measure of complexity, which is characterized by the difference in slopes of $u_{j}(\cdot)$ between adjacent subintervals \citep{ref7}. Since $[\underline{g_{j}},\overline{g_{j}}]$ is equally divided into $\lambda_{j}$ subintervals, the difference in slopes between adjacent subintervals is determined by the value difference between these  subintervals. Therefore, the following constraints are constructed:
\begin{equation}
  \begin{cases}
[MC]:\triangle u_{j}^{s}\geq0,\ s=1,2,\cdots,\lambda_{j},\ j\in M,\\
[NC]:\begin{cases}
 u_{j}(g_{j}^{0})=0,\ j\in M,\\
 \sum\limits_{j=1}^{m}\delta_{j}\times\sum\limits_{s=1}^{\lambda_{j}}\triangle u_{j}^{s}=1,
 \end{cases}\\
[SC]:|\triangle u_{j}^{s+1}-\triangle u_{j}^{s}|\leq\varphi,\ s=1,\cdots,\lambda_{j}-1,\ j\in M,\\
\delta_{j}=\begin{cases}
1,~the~DM~evaluates~the~alternatives~based~on~c_{j},\\
0,~otherwise,
\end{cases}j\in M,
 \end{cases}
\end{equation}
where the constraints $[SC]$ are used to control the difference in slopes between adjacent subintervals.

In fact, the complexity of $u_{\ast}(\cdot)$ is related to the number of marginal value functions that constitute $u_{\ast}(\cdot)$, in addition to the degree to which the marginal value functions deviate from linearity. The fewer marginal value functions composing $u_{\ast}(\cdot)$, the fewer model parameters, and the lower the model complexity. As a result, we can control the complexity of value functions by minimizing the following regularization term:
\begin{equation}\label{conplexity}
  \Omega(u_{\ast})=p\times\varphi+\sum\limits_{j=1}^{m}\delta_{j},
\end{equation}
where $\varphi$ measures the extent to which marginal value functions deviate from linearity and $\sum\limits_{j=1}^{m}\delta_{j}$ is the number of marginal value functions that compose value functions. The parameter $p>0$ is a constant used to strike a trade-off between the two. When minimizing the regularization term as the objective function (model (\ref{regularization-2})) or as a part of the objective function (model (\ref{regularization-1})), if parameter $p$ is set to a larger value, the resulting value function tends to be more linear, but this may lead to the selection of more criteria. The choice of the value of $p$ depends on the preference of the relevant personnel towards balancing these two complexities. If a closer approximation to a linear value function is desired, a larger $p$ should be set; if the goal is to minimize the number of selected criteria, a smaller $p$ should be chosen.

\section{The criteria selection model}\label{criteria selection model}
Due to limitations in cognitive ability, DMs are unable to evaluate alternatives based on all criteria when the number of criteria is large. In such a case, the preference information given by DMs is the result of their judgments based on the criteria they are most concerned with. Given such preference information, we need to infer not only the value function that approximates the DM's preference system, but also the criteria considered by the DM to draw his (her) judgments. In this section, we discuss this problem using preference disaggregation technique and regularization theory in two scenarios, where the DM's preference information is consistent and inconsistent, respectively.

\subsection{Scenario of inconsistent preference information}\label{main-model-1}
When preference information is inconsistent, according to Remark \ref{aa}, the larger set $\mathcal{C}_{t}$ is, the smaller may be the empirical error of $\mathcal{C}_{t}$-based value functions, but then these value functions have a higher complexity. Therefore, we need to make a trade-off between the empirical error and the generalization error of value functions. To achieve this, the following criteria selection model is constructed:

\begin{align}\label{regularization-1}
\begin{array}{ll}
 \text{minimize}   & C\times\sum\limits_{a_{i}\in A^{R}}(\sigma_{+}(a_{i})+\sigma_{-}(a_{i}))+(p\times\varphi+\sum\limits_{j=1}^{m}\delta_{j})\\
 \text{subject to} & \begin{cases}
 [MC]:\triangle u_{j}^{s}\geq0,\ s=1,2,\cdots,\lambda_{j},\ j\in M,\\
[NC]:\begin{cases}
 u_{j}(g_{j}^{0})=0,\ j\in M,\\
 \sum\limits_{j=1}^{m}\delta_{j}\times\textbf{V}_{j}^{\textrm{T}}\times\textbf{I}=1,
 \end{cases}\\
[SC]:|\triangle u_{j}^{s+1}-\triangle u_{j}^{s}|\leq\varphi,\ s=1,\cdots,\lambda_{j}-1,\ j\in M,\\
 E^{A^{R}}:\begin{cases}
 u_{\ast}(a_{p})+\sigma_{+}(a_{p})-\sigma_{-}(a_{p})=u_{\ast}(a_{q})+\sigma_{+}(a_{q})-\sigma_{-}(a_{q}),\\ ~\mbox{for}~\ a_{p},a_{q}\in A^{R}\ ~\&~\ a_{p}\sim a_{q},\\
 u_{\ast}(a_{p})+\sigma_{+}(a_{p})-\sigma_{-}(a_{p})\geq u_{\ast}(a_{q})+\sigma_{+}(a_{q})-\sigma_{-}(a_{q})+\varepsilon,\\~\mbox{for}~\ a_{p},a_{q}\in A^{R}\ ~\&~\ a_{p}\succ a_{q},\\
 \sigma_{+}(a_{p}),\sigma_{+}(a_{q}),\sigma_{-}(a_{p}),\sigma_{-}(a_{q})\geq0,~\mbox{for}~\ a_{p},a_{q}\in A^{R},
 \end{cases}\\
 \delta_{j}=\begin{cases}
1,~the~DM~evaluates~the~alternatives~based~on~c_{j},\\
0,~otherwise,
\end{cases}j\in M,\\
 \varepsilon>0.
\end{cases}
\end{array}
\end{align}
The decision variables of this model are: $\varphi,\ \{\textbf{V}_{j},\ j\in M\},\ \{\delta_{j},\ j\in M\}$, $\{\sigma_{+}(a_{i}),\sigma_{-}(a_{i}),\ a_{i}\in A^{R}\}$. The objective function is a general regularized loss function, which consists of two parts: empirical error $\sum\limits_{a_{i}\in A^{R}}(\sigma_{+}(a_{i})+\sigma_{-}(a_{i}))$ and generalization error $p\times\varphi+\sum\limits_{j=1}^{m}\delta_{j}$. The constant $C$ is used to balance these two parts. The value of $C$ is related to the preferences of the relevant personnel (DMs or decision analysts). If they wish that the value function faithfully reflects the DM's preference information, they can set a larger $C$. In this case, model (\ref{regularization-1}) will prioritize minimizing the empirical error, but this may lead to the selection of too many criteria and the constructed value function deviating too much from linearity. If the value of $C$ is small, model (\ref{regularization-1}) will reduce the complexity of the model by sacrificing the fitting ability of the value function to preference information, meaning that the constructed value function will be close to linear and involve fewer criteria. In extreme cases, only one criterion will be selected. A possible way to determine a reasonable value of $C$ is to preset several different values for $C$ according to the preferences of the relevant personnel. The calculation results and corresponding suggestions of model (\ref{regularization-1}) under these values of $C$ are then presented to the DM, who can determine an appropriate value of $C$.

Given the optimal solution of model (\ref{regularization-1}), $\{c_{j}|\delta_{j}=1\}$ are the criteria that the DM considered when making judgments of reference alternatives, and the corresponding value function is $u(\cdot)=\sum\limits_{\{c_{j}|\delta_{j}=1\}}\textbf{V}_{j}^{\textrm{T}}\times \textbf{A}_{j}(\cdot)$. If there is a need to limit the number of selected criteria, the constraint $\sum\limits_{j=1}^{m}\delta_{j}\leq L$ can be added. Due to the presence of $\delta_{j}\times \textbf{V}_{j}^{\textrm{T}}$, model (\ref{regularization-1}) is a nonlinear mixed-integer program, which is hard to solve. Therefore, we transform this model into an equivalent linear form through mathematical transformation.

Let $z_{j}^{s}=\delta_{j}\times\triangle u_{j}^{s}$,\ \ $\textbf{Z}_{j}^{\textrm{T}}=(z_{j}^{1},z_{j}^{2},\cdots,z_{j}^{\lambda_{j}})^{\textrm{T}}$, and add the following constraints to model (\ref{regularization-1}):
\begin{equation}\label{transformation}
 \begin{cases}
 z_{j}^{s}\leq\triangle u_{j}^{s},\\
 z_{j}^{s}\geq\triangle u_{j}^{s}-(1-\delta_{j}),\\
 \rho\times\delta_{j}\leq z_{j}^{s}\leq\delta_{j},
 \end{cases} s=1,2,\cdots,\lambda_{j}.
\end{equation}
From Eq. (\ref{transformation}), we can see that when $\delta_{j}=1$, $z_{j}^{s}=\triangle u_{j}^{s}$, and when $\delta_{j}=0$, $z_{j}^{s}=0$.  After such transformation, model (\ref{regularization-1}) takes the following form:
\begin{align}\label{regularization-1*}
\begin{array}{ll}
 \text{minimize}   & C\times\sum\limits_{a_{i}\in A^{R}}(\sigma_{+}(a_{i})+\sigma_{-}(a_{i}))+(p\times\varphi+\sum\limits_{j=1}^{m}\delta_{j})\\
 \text{subject to} & \begin{cases}
 [MC]:\triangle u_{j}^{s}\geq0,\ s=1,2,\cdots,\lambda_{j},\ j\in M,\\
[NC]:\begin{cases}
 u_{j}(g_{j}^{0})=0,\ j\in M,\\
 \sum\limits_{j=1}^{m}\textbf{Z}_{j}^{\textrm{T}}\times\textbf{I}=1,
 \end{cases}\\
[SC]:|\triangle u_{j}^{s+1}-\triangle u_{j}^{s}|\leq\varphi,\ s=1,\cdots,\lambda_{j}-1,\ j\in M,\\
 E^{A^{R}}:\begin{cases}
 u_{z\ast}(a_{p})+\sigma_{+}(a_{p})-\sigma_{-}(a_{p})=u_{z\ast}(a_{q})+\sigma_{+}(a_{q})-\sigma_{-}(a_{q}),\\~\mbox{for}~\ a_{p},a_{q}\in A^{R}\ ~\&~ \ a_{p}\sim a_{q},\\
 u_{z\ast}(a_{p})+\sigma_{+}(a_{p})-\sigma_{-}(a_{p})\geq u_{z\ast}(a_{q})+\sigma_{+}(a_{q})-\sigma_{-}(a_{q})+\varepsilon,\\~\mbox{for}~\ a_{p},a_{q}\in A^{R}\ ~\&~\ a_{p}\succ a_{q},\\
 \sigma_{+}(a_{p}),\sigma_{+}(a_{q}),\sigma_{-}(a_{p}),\sigma_{-}(a_{q})\geq0,~\mbox{for}~\  a_{p},a_{q}\in A^{R},
 \end{cases}\\
 [LC]:\begin{cases}
 \textbf{Z}_{j}^{\textrm{T}}\leq\textbf{V}_{j}^{\textrm{T}},\\
 \textbf{Z}_{j}^{\textrm{T}}\geq\textbf{V}_{j}^{\textrm{T}}-(1-\delta_{j})\times\textbf{I},\\
 \rho\times\delta_{j}\times\textbf{I}\leq\textbf{Z}^{\textrm{T}}_{j}\leq\delta_{j}\times\textbf{I},
 \end{cases} j\in M,\\
 \delta_{j}=\begin{cases}
1,~the~DM~evaluates~the~alternatives~based~on~c_{j},\\
0,~otherwise,
\end{cases}j\in M,\\
 \varepsilon>0.
\end{cases}
\end{array}
\end{align}
where $u_{z\ast}(\cdot)=\sum\limits_{j=1}^{m}\textbf{Z}_{j}^{\textrm{T}}\times \textbf{A}_{j}(\cdot)$, and $\textbf{Z}_{j}^{\textrm{T}}\leq\textbf{V}_{j}^{\textrm{T}}$ indicates that the corresponding elements of the two vectors satisfy the inequality. In addition to $\varphi,\ \{\textbf{V}_{j}^{\textrm{T}},\ j\in M\},\ \{\delta_{j},\ j\in M\}$, $\{\sigma_{+}(a_{i}),\sigma_{-}(a_{i}),\ a_{i}\in A^{R}\}$, model (\ref{regularization-1*}) also adopts $\{\textbf{Z}^{\textrm{T}}_{j},\ j\in M\}$ as decision variables. This transformation only adds $m$ variables and $4\times\sum\limits_{j=1}^{m}\lambda_{j}$ linear constraints. Model (\ref{regularization-1*}) is a linear mixed-integer programming. Many mature solvers can efficiently solve this model, such as Gurobi, COPT, MATLAB and Lingo.

\begin{proposition}\label{main-proof}
Model (\ref{regularization-1}) and model (\ref{regularization-1*}) are equivalent.
\end{proposition}
\begin{proof}
See \ref{A2}.
\end{proof}

\begin{proposition}\label{well-define}
The value function determined by model (\ref{regularization-1*}) is non-degenerate.
\end{proposition}
\begin{proof}
See \ref{A3}.
\end{proof}

Proposition \ref{well-define} ensures that the performance of alternatives on the criteria selected by model (\ref{regularization-1*}) is valuable to DMs, i.e., they indeed evaluate alternatives based on these criteria. If $p=0$, the value function determined by model (\ref{regularization-1*}) has the following characteristic.
\begin{remark}\label{cc}
Let $p=0$. Assume that the value function determined by model (\ref{regularization-1*}) is $\overline{u(\cdot)}$, and it consists of $k$ marginal value functions. Then, $\overline{u(\cdot)}$  is the value function with the smallest empirical error among all value functions based on criteria subsets that include $k$ criteria.
\end{remark}

Model (\ref{regularization-1*}) serves dual function of preference learning and criteria selection. By solving model (\ref{regularization-1*}), we can not only infer which criteria the DM uses to evaluate alternatives but also determine the value function that represents the DM's preference system. From the perspective of preference learning, model (\ref{regularization-1*}) uses parameter $C$ to balance the fitting ability and complexity of value functions. The complexity has two meanings, one is the degree of deviation from linearity, and the other is the number of marginal value functions. The learned value function, while maintaining a certain level of fitting ability for the DM's preference information, takes a simpler form. Such value function can be considered as ``simple but not simpler\footnote{``Everything should be made as simple as possible, but not simpler''--commonly attributed to Albert Einstein}''. From the perspective of feature selection, model (\ref{regularization-1*}) serves as an embedded feature selection method. It evaluates the quality of feature subsets based on the empirical error and generalization error of value functions, and completes the learning of the value function during the process of feature selection. Model (\ref{regularization-1*}) takes a few holistic judgments of reference alternatives as an input, and DMs can explore the learned value function by altering or enriching the preference information. Furthermore, the criteria selection method based on 0-1 programming adopts to different types of domain knowledge provided by the DMs. These characteristics make preference learning constructive and increase the DMs' confidence in the decision aiding.

\subsection{Scenario of consistent preference information}\label{main-model-2}
If the preference information is consistent, there exists a $\mathcal{C}$-based compatible value function capable of restoring this preference information. The fact that $\mathcal{C}$ is the supporting criteria set implies the possibility of existence of smaller relevant criteria subsets.
At this point, there are several interesting points worth of our attention:
\begin{itemize}
  \item The supporting criteria set with minimum number of relevant criteria. Such supporting criteria set is the simplest relative to the preference information.
  \item The common part of streamlined supporting criteria sets. Such common criteria are inevitably taken into account by DMs when evaluating reference alternatives.
  \item Criteria that are not included in any streamlined supporting criteria set. Such criteria are not taken into account by DMs when evaluating reference alternatives.
\end{itemize}

To identify the simplest supporting criteria set, we construct the following optimization model:
\begin{align}\label{regularization-2}
\begin{array}{ll}
 \text{minimize}   & p\times\varphi+\sum\limits_{j=1}^{m}\delta_{j}\\
 \text{subject to} & \begin{cases}
 [MC]:\triangle u_{j}^{s}\geq0,\ s=1,2,\cdots,\lambda_{j},\ j\in M,\\
[NC]:\begin{cases}
 u_{j}(g_{j}^{0})=0,\ j\in M,\\
 \sum\limits_{j=1}^{m}\delta_{j}\times\textbf{V}_{j}^{\textrm{T}}\times\textbf{I}=1,
 \end{cases}\\
[SC]:|\triangle u_{j}^{s+1}-\triangle u_{j}^{s}|\leq\varphi,\ s=1,\cdots,\lambda_{j}-1,\ j\in M,\\
 E^{A^{R}}:\begin{cases}
 \sum\limits_{j=1}^{m}\delta_{j}\times \textbf{V}_{j}^{\textrm{T}}\times \textbf{A}_{j}(a_{p})=\sum\limits_{j=1}^{m}\delta_{j}\times \textbf{V}_{j}^{\textrm{T}}\times \textbf{A}_{j}(a_{q}),\\~\mbox{for}~\ a_{p},a_{q}\in A^{R}\ ~\&~ \ a_{p}\sim a_{q},\\
 \sum\limits_{j=1}^{m}\delta_{j}\times \textbf{V}_{j}^{\textrm{T}}\times \textbf{A}_{j}(a_{p})\geq \sum\limits_{j=1}^{m}\delta_{j}\times \textbf{V}_{j}^{\textrm{T}}\times \textbf{A}_{j}(a_{q})+\varepsilon,\\~\mbox{for}~\ a_{p},a_{q}\in A^{R}\ ~\&~ \ a_{p}\succ a_{q},
 \end{cases}\\
  \delta_{j}=\begin{cases}
1,~the~DM~evaluates~the~alternatives~based~on~c_{j},\\
0,~otherwise,
\end{cases}j\in M,\\
 \varepsilon>0.
\end{cases}
\end{array}
\end{align}
Model (\ref{regularization-2}) performs criteria selection and value function construction by weighing the extent to which the value function deviates from linearity and the number of marginal value functions, and the supporting criteria set determined by this model is simple as long as $p$ is sufficiently small. Similarly, model (\ref{regularization-2}) can be transformed into a linear mixed-integer programming using the transformation introduced in Section \ref{main-model-1}:
\begin{align}\label{regularization-2*}
\begin{array}{ll}
 \text{minimize}   & p\times\varphi+\sum\limits_{j=1}^{m}\delta_{j}\\
 \text{subject to} & \begin{cases}
 [MC]:\triangle u_{j}^{s}\geq0,\ s=1,2,\cdots,\lambda_{j},\ j\in M,\\
[NC]:\begin{cases}
 u_{j}(g_{j}^{0})=0, j\in M,\\
 \sum\limits_{j=1}^{m}\textbf{Z}_{j}^{\textrm{T}}\times\textbf{I}=1,
 \end{cases}\\
[SC]:|\triangle u_{j}^{s+1}-\triangle u_{j}^{s}|\leq\varphi,\ s=1,\cdots,\lambda_{j}-1,\ j\in M,\\
 E^{A^{R}}:\begin{cases}
\sum\limits_{j=1}^{m}\textbf{Z}_{j}^{\textrm{T}}\times \textbf{A}_{j}(a_{p})=\sum\limits_{j=1}^{m}\textbf{Z}_{j}^{\textrm{T}}\times \textbf{A}_{j}(a_{q}), \\ ~\mbox{for}~\ a_{p},a_{q}\in A^{R}\ ~\&~ \ a_{p}\sim a_{q},\\
 \sum\limits_{j=1}^{m}\textbf{Z}_{j}^{\textrm{T}}\times \textbf{A}_{j}(a_{p})\geq \sum\limits_{j=1}^{m}\textbf{Z}_{j}^{\textrm{T}}\times \textbf{A}_{j}(a_{q})+\varepsilon,\\ ~\mbox{for}~\ a_{p},a_{q}\in A^{R}\ ~\&~ \ a_{p}\succ a_{q},
 \end{cases}\\
 [LC]:\begin{cases}
 \textbf{Z}_{j}^{\textrm{T}}\leq\textbf{V}_{j}^{\textrm{T}},\\
 \textbf{Z}_{j}^{\textrm{T}}\geq\textbf{V}_{j}^{\textrm{T}}-(1-\delta_{j})\times\textbf{I},\\
 \rho\times\delta_{j}\times\textbf{I}\leq\textbf{Z}^{\textrm{T}}_{j}\leq\delta_{j}\times\textbf{I},
 \end{cases}j\in M,\\
 \delta_{j}=\begin{cases}
1,~the~DM~evaluates~the~alternatives~based~on~c_{j},\\
0,~otherwise,
\end{cases}j\in M,\\
 \varepsilon>0.
\end{cases}
\end{array}
\end{align}
The criteria corresponding to $\delta_{j}=1$  form the supporting criteria set $\mathcal{C}_{k}$ for preference information.
\begin{proposition}\label{bb}
The value function determined by model (\ref{regularization-2*}) is non-degenerate.
\end{proposition}
\begin{proof}
The proof of this proposition is analogous to that of Proposition \ref{well-define} and is omitted here.
\end{proof}
\begin{remark}\label{dd}
Assume that the value function determined by model (\ref{regularization-2*}) is $\overline{u(\cdot)}$, and it consists of $k$ marginal value functions. Then $\overline{u(\cdot)}$ is the most linear of all value functions based on criteria subsets containing $k$ criteria.
\end{remark}

It is evident that the simplest supporting criteria set is streamlined. Whenever a streamlined supporting criteria set $\mathcal{C}_{k}$ is identified, we add the constraint $\sum\limits_{\{j|c_{j}\in \mathcal{C}_{k}\}}\delta_{j}\leq|\mathcal{C}_{k}|-1$ to model (\ref{regularization-2*}) and re-solve it. By iterating in this way, all possible streamlined supporting sets can be identified.

\begin{proposition}\label{reduction-set}
The supporting criteria sets determined by the above method are all streamlined.
\end{proposition}
\begin{proof}
See \ref{A4}.
\end{proof}
Any criteria subset that contains a streamlined supporting criteria set is also a supporting criteria set.  From the perspective of criteria selection, we do not focus on such supporting criteria sets as they contain unnecessary criteria. From a modeling perspective, such criteria subsets are valuable as the value function based on a larger criteria subset may be closer to linearity (this will be illustrated later in the case study).

Each streamlined supporting criteria set has the potential to be considered by DMs when evaluating reference alternatives. It reflects the DM's judgment policy regarding the performance of alternatives. Based on all supporting criteria sets, we can define the relevance of criteria, core criteria, and redundant criteria.
\begin{definition}
Assume that $\mathcal{C}_{k},k=1,\cdots,S$, are all streamlined supporting criteria sets for preference information. Then,
\begin{enumerate}
  \item The relevance of $c_{j}$, denoted by $R(c_{j})$, is equal to the number of streamlined supporting criteria sets that include $c_{j}$, i.e., $R(c_{j})=|\{\mathcal{C}_{k}|c_{j}\in\mathcal{C}_{k},\ k=1,\cdots,S\}|$.
  \item If $\bigcap\limits_{k=1}^{S}\mathcal{C}_{k}\neq\emptyset$, the set $\mathcal{C}_{core}=\bigcap\limits_{k=1}^{S}\mathcal{C}_{k}$ is called the core criteria set, and the criteria in it are referred to as core criteria. Obviously, the core criteria have the maximum relevance.
  \item The set $\mathcal{C}\setminus\bigcup_{k=1}^{S}\mathcal{C}_{k}$ is referred to as redundant criteria set, and the criteria in it are called redundant criteria. Redundant criteria are not included in any streamlined supporting criteria set, and their relevance is $0$.
\end{enumerate}
\end{definition}
The larger $R(c_{j})$ is, the more streamlined supporting criteria sets that contain $c_{j}$, indicating that DMs are more likely to rely on this criterion when evaluating reference alternatives. If we need to select one of many streamlined supporting criteria sets, the sum of relevance of all criteria  in a streamlined supporting criteria set can serve as an indicator. Both core criteria and redundant criteria are robust and have a clear interpretation. Core criteria are those that DMs are certain to consider when evaluating reference alternatives, while redundant criteria are those they are unlikely to rely on. Both types of criteria contribute to understanding of the DM's preference system, thereby enabling the analyst to provide a better decision support.

If preference information is consistent, there exists a compatible value function that restores this information. Therefore, the objective function of model (\ref{regularization-2}) does not include the empirical error of value functions. In fact, in some decision scenarios, we can still measure the fitting ability of value functions. For instance, \citet{ref10} use the maximum difference of comprehensive value between two alternatives with preference relation (``$\succ$'') as the fitting ability of value functions. This approach emphasizes the preference information, reproducing them boldly and robustly. In this case, the general regularization loss function is $\varepsilon+(m\times\varphi+C\times\sum\limits_{j=1}^{m}\delta_{j})$.

Model (\ref{regularization-2*}) aims to find a value function that is completely consistent with the DM's preference information. This may raise two issues:
\begin{itemize}
  \item \textbf{Criteria overly selected.} Model (\ref{regularization-2*}) only determines the supporting criteria set for preference information, and the smallest supporting criteria set may include too many criteria. For example, assuming that the smallest supporting criteria set contains $6$ criteria, this suggests that constructing a value function based on at least $6$ criteria is required to restore the preference information provided by the DM. However, we are interested in finding out which five criteria the DM would most likely rely on if limited to $5$, and similarly, which four criteria they would consider if restricted to $4$. This information is crucial for analyzing the DM's preference system and serves as the starting point for a more elaborate decision aiding process.
  \item \textbf{Generalization ability weakened.} The DM may have a cognitive bias, resulting in some ``noisy'' preference information, i.e., some preference information which does not align with the DM's preference system. Consequently, the value function determined by model (\ref{regularization-2*}) may be influenced by this noise, leading to deviation from the DM's preference system and a decrease in generalization ability. This ``noise'' may originate from irrational behavior by the DM, whereby alternatives are evaluated not only on the scores of the considered criteria, but also on extraneous factors. This can lead to preference reversals, where the preference relation of two alternatives changes due to the influence of irrelevant factors \citep{buchong8}.
\end{itemize}
To avoid the aforementioned issues, we can allow some preference information to remain unrestored and assume a limited estimation errors, subsequently using Model (\ref{regularization-1*}).

\section{An example}\label{case study}
In this section, we validate the effectiveness of the proposed criteria selection method on a hypothetical problem of selecting a green supplier.

In today's era, the realities of the shortage of resources and environmental pollution require enterprises to not only pursue economic benefits but also maintain environment friendliness in the process of development. Green Supply Chain Management (GSCM) is a new management mode that balances the economic benefits and environment sustainable development, becoming an important topic in modern enterprise production and operation management \citep{ref50}. The core idea of GSCM is to ensure that environmental concerns are addressed at every stage of the product's life-cycle, from the design of products and sourcing of raw materials, through manufacturing and logistics, to the end-of-life disposal or recycling of products. The green supplier evaluation and selection is one of the core parts of GSCM, which directly affects the environment protection performance of enterprises. They can all be considered as a multi-criteria decision-making problem involving many conflicting criteria \citep{ref51}.

Let us consider an automobile manufacturing enterprise aiming to select a supplier to purchase key components. There are 10 candidates evaluated by various criteria, including green product innovation ($g_{1}$), green image ($g_{2}$), use of environmentally friendly technology ($g_{3}$), resource consumption ($g_{4}$), green competencies ($g_{5}$), environment management ($g_{6}$), quality management ($g_{7}$), total product life cycle cost ($g_{8}$), pollution production ($g_{9}$) and staff environmental training ($g_{10}$) \citep{ref50}. Without loss of generality, it is assumed that all these criteria are benefit-type and have a evaluation scale of $[0,1]$. The performance of considered suppliers on these criteria is shown in Table \ref{Decision Martix}. They can be all considered as reference alternatives. An expert is invited to provide some holistic judgments about these suppliers, as shown in Table \ref{Preference information}. In this decision problem, we want to know what criteria the DM relies on to evaluate the suppliers and how these criteria are aggregated to get the DM's preference model.
\begin{table}[!ht]
\centering
\caption{The performance of suppliers}
\scalebox{0.9}{\begin{tabular}{ccccccccccc}\toprule
   Suppliers&	$g_{1}$ & $g_{2}$ & $g_{3}$ & $g_{4}$ & $g_{5}$ & $g_{6}$ & $g_{7}$ & $g_{8}$ & $g_{9}$ & $g_{10}$\\ \midrule
$a_{1}$&	 $0.93$ & $0.38$ & $0.33$ & $0.22$ & $0.98$ & $0.34$ & $0.18$ & $0.92$ & $0.85$ & $0$\\
$a_{2}$&	 $0.75$ & $0.82$ & $0.43$& $0.05$ & $0.27$ & $0.16$ & $0.8$& $0.23$& $0.68$& $0.04$\\
$a_{3}$&	 $0.95$ & $0.44$ & $0.61$ & $0.66$ & $0.83$ & $0.45$ & $0.22$ & $0.95$ & $0.11$ & $0.37$\\
$a_{4}$&	 $0.07$ & $0.5$ & $0.81$ & $0.62$ & $0.53$ & $0.08$ & $0.17$ & $0.76$ & $0.85$ & $0.1$\\
$a_{5}$&	 $0.8$ & $0.77$ & $0.54$ & $0.05$ & $0.1$ & $0.2$ & $0.89$ & $0.47$ & $0.53$ & $0.28$\\
$a_{6}$&	 $0.65$ & $0.1$ & $0.91$ & $0.3$ & $0.88$ & $0.33$ & $0.42$& $1$ & $0.47$ & $0.38$\\
$a_{7}$&	 $1$ & $0.28$ & $0.9$ & $0.55$ & $0.4$ & $0.98$ & $0.8$& $0.79$ & $0.18$ & $0.21$\\
$a_{8}$&	 $0.5$ & $0.04$ & $0.67$ & $0.81$ & $0.9$ & $0.7$ & $0.26$ & $0.95$ & $0.71$ & $0.24$\\
$a_{9}$&  $0.46$ & $0.62$ & $0.98$ & $0.64$ & $0.33$ & $1$ & $0.56$& $0.6$ & $0.05$ & $0.23$\\
$a_{10}$& $0.9$ & $0.02$ & $0.88$ & $0.31$ & $0.52$ & $0.63$ & $0.79$& $0.83$ & $0.77$ & $0.16$\\
   \bottomrule
     \end{tabular}}\label{Decision Martix}
\end{table}
\begin{table}[!ht]
\centering
\caption{Preference information}
\scalebox{0.9}{\begin{tabular}{cccccccc}\toprule
   $a_{5}\succ a_{1}$&	$a_{4}\succ a_{3}$ & $a_{7}\succ a_{8}$ & $a_{7}\succ a_{9}$ & $a_{3}\succ a_{10}$ & $a_{9}\succ a_{10}$ & $a_{7}\succ a_{6}$ &
   $a_{4}\succ a_{7}$\\
   \bottomrule
 \end{tabular}}\label{Preference information}
\end{table}

\subsection{Compared to the UTA method without considering criteria selection.}
Assuming that the evaluation scale is equally divided into $5$ sub-intervals ($4$ breakpoints), each marginal value function is determined by $5$ parameters:
\begin{equation}
  u_{j}(g_{j}(a_{i}))=\sum\limits_{s=1}^{5}\triangle u_{j}^{s}\times\delta_{g_{j}}^{s}(a_{i}).
\end{equation}
The preference information provided by the DM is converted into constraints of model (\ref{UTA}). The calculation results show that the preference information is consistent. There exist multiple value functions constructed using all criteria that can restore the DM's preference information. We select one of these compatible value functions to represent the DM's preference system, based on the principle of maximizing the difference of comprehensive value between two alternatives being in preference relation ``$\succ$'' \citep{ref10}. Figure \ref{basline-figure} is the graph of the selected value function. Although this value function is constructed based on all criteria, only three marginal value functions are non-zero, and all of them deviate significantly from linearity. Table \ref{basline-score} presents the comprehensive scores of the suppliers, from which we can derive the ranking of these suppliers as follows: $$a_{2}\succ a_{5}\succ a_{4}\succ a_{1}\succ a_{7}\succ a_{3}\succ a_{8}\succ a_{9}\succ a_{6}\succ a_{10}.$$
\begin{figure}[!ht]
  \centering
  \includegraphics[scale=0.3]{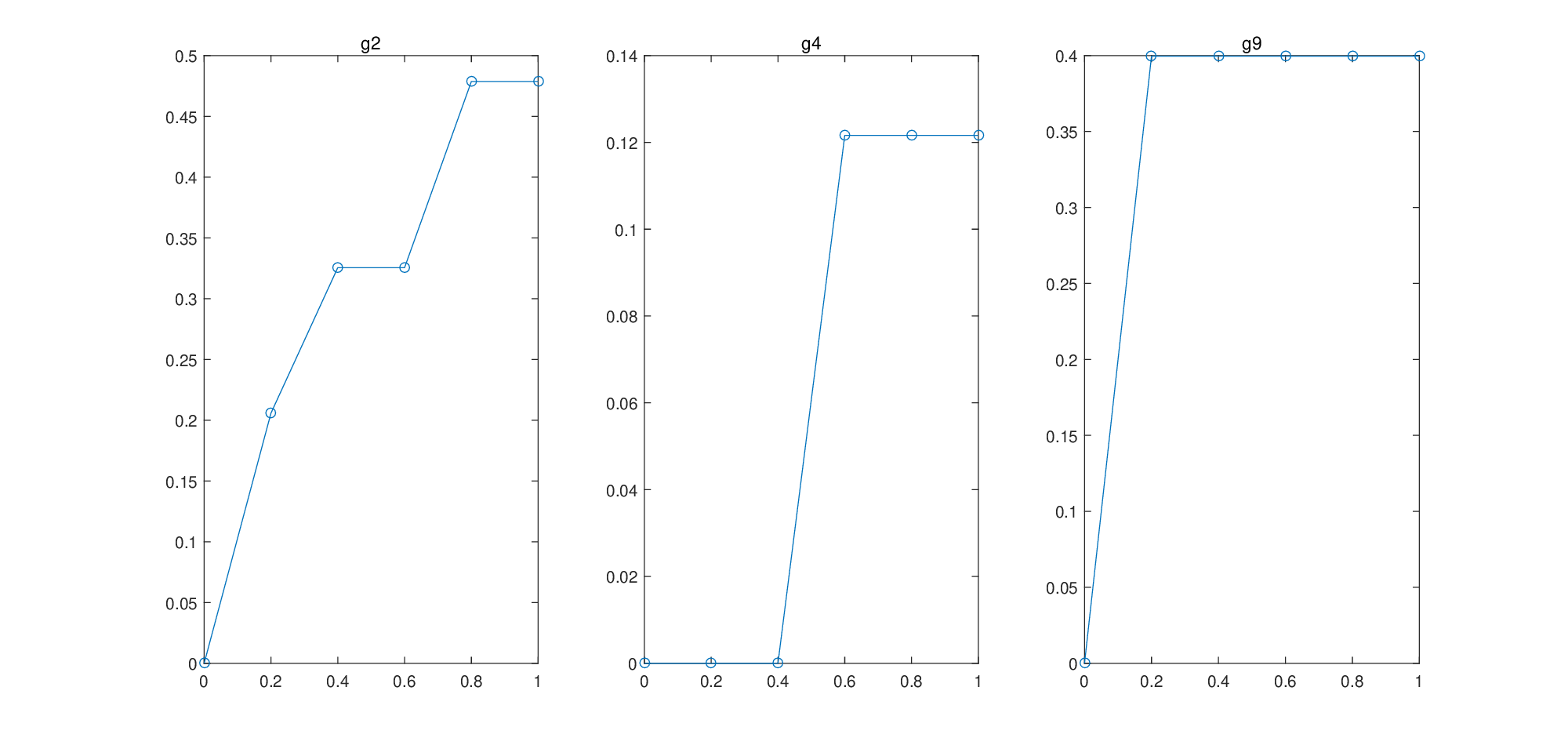}
  \caption{The value function determined by the principle of maximizing the difference of comprehensive value between two alternatives being in relation ``$\succ$''.}
  \label{basline-figure}
\end{figure}
\begin{table}[!ht]
\centering
\caption{Comprehensive scores of suppliers}
\scalebox{0.8}{\begin{tabular}{ccccccccccc}\toprule
  $a_{1}$ & $a_{2}$ & $a_{3}$ & $a_{4}$ & $a_{5}$ & $a_{6}$ & $a_{7}$ & $a_{8}$ & $a_{9}$ & $a_{10}$ \\ \midrule
  $0.7131$ & $0.8783$ & $0.6669$ & $0.8467$ & $0.8553$ & $0.5025$ & $0.7045$ & $0.5624$ & $0.5623$ & $0.4202$\\
   \bottomrule
 \end{tabular}}\label{basline-score}
\end{table}

Next, we select criteria and construct the value function according to the method proposed in this paper. With different values of $p$, model (\ref{regularization-2*}) is used to infer which criteria the DM relies on to evaluate the suppliers and how these criteria are aggregated. Table (\ref{Calculation results}) shows the criteria determined by model (\ref{regularization-2*}) and the deviation from linearity of the corresponding value function. When $p\leq 10$, the criteria determined by model (\ref{regularization-2*}) are $\{g_{2},g_{9}\}$. This suggests that the DM evaluates suppliers based on at least two criteria. Moreover, for all values of $p\leq 10$, the value function determined by model (\ref{regularization-2*}) remains unchanged. When $20\leq p\leq180$,  criteria set $\{g_{1},g_{2},g_{9}\}$ and the value function determined by model (\ref{regularization-2*}) also remain constant. When $p\geq200$, a completely linear additive value function can be constructed based on $\{g_{2},g_{7},g_{8},g_{9}\}$ to restore the preference information provided by the DM. It can be observed that as $p$ increases, model (\ref{regularization-2*}) selects more criteria, and the learned value function becomes closer to linear. Furthermore, model (\ref{regularization-2*}) is robust in large ranges of parameter $p$. When $p$ varies within a certain range, the criteria and the value function determined by model (\ref{regularization-2*}) do not change. In fact, regardless of the criteria subset identified by model (\ref{regularization-2*}), the output value function must be the one that is closest to linear among all compatible value functions constructed based on the corresponding criteria subset.
\begin{table}[!ht]
\centering
\caption{Calculation results of model (18)}
\scalebox{0.9}{\begin{tabular}{cccc}\toprule
   $p$ &	The selected criteria &$\sum\limits_{j=1}^{m}\delta_{j}$ &$\varphi$ \\ \midrule
$<1$&	 $g_{2},g_{9}$ & 2 & 0.071\\
$1$    &     $g_{2},g_{9}$ & 2 & 0.071\\
$2$    &     $g_{2},g_{9}$ & 2 & 0.071\\
$4$    &     $g_{2},g_{9}$ & 2 & 0.071\\
$6$    &     $g_{2},g_{9}$ & 2 & 0.071\\
$8$    &     $g_{2},g_{9}$ & 2 & 0.071\\
$10$&	 $g_{2},g_{9}$ & 2 & 0.071\\
$20$&	 $g_{1},g_{2},g_{9}$ & 3 & 0.005\\
$40$&	 $g_{1},g_{2},g_{9}$ & 3 & 0.005\\
$60$&	 $g_{1},g_{2},g_{9}$ & 3 & 0.005\\
$80$&	 $g_{1},g_{2},g_{9}$ & 3 & 0.005\\
$100$&	 $g_{1},g_{2},g_{9}$ & 3 & 0.005\\
$120$&	 $g_{1},g_{2},g_{9}$ & 3 & 0.005\\
$140$&	 $g_{1},g_{2},g_{9}$ & 3 & 0.005\\
$160$&	 $g_{1},g_{2},g_{9}$ & 3 & 0.005\\
$180$&	 $g_{1},g_{2},g_{9}$ & 3 & 0.005\\
$200$&	 $g_{2},g_{7},g_{8},g_{9}$ & 4 & $0$\\
$>200$&   $g_{2},g_{7},g_{8},g_{9}$ & 4 & $0$\\
   \bottomrule
 \end{tabular}}\label{Calculation results}
\end{table}

Figures (\ref{p10})$-$(\ref{p200}) show the graphs of value functions determined by model (\ref{regularization-2*}). Table \ref{cridit risk} contains the comprehensive scores of suppliers. It can be observed that the comprehensive scores of suppliers are consistent with the preference information provided by the DM. Based on the comprehensive scores of the suppliers in Table \ref{cridit risk}, the rankings of these suppliers are as follows: 
\begin{itemize}
  \item $p=10:a_{2}\succ a_{5}\succ a_{4}\succ a_{1}\succ a_{3}\succ a_{7}\succ a_{6}\sim a_{9}\succ a_{8}\succ a_{10}$;
  \item $p=100:a_{2}\succ a_{5}\succ a_{1}\succ a_{4}\succ a_{3}\succ a_{7}\succ a_{9}\succ a_{10}\succ a_{6}\succ a_{8}$;
  \item $p=200:a_{5}\succ a_{1}\succ a_{2}\succ a_{4}\succ a_{3}\sim a_{7}\succ a_{9}\succ a_{10}\succ a_{6}\succ a_{8}$.
\end{itemize}
Although the ranking of suppliers differs for different values of $p$, the differences are not significant. For instance, the top 4 suppliers are always $a_{1}$, $a_{2}$, $a_{4}$ and $a_{5}$. And without considering the selection of criteria, the top four suppliers are also $a_{1}$, $a_{2}$, $a_{4}$ and $a_{5}$. 
\begin{figure}[!ht]
  \centering
  \includegraphics[scale=0.4]{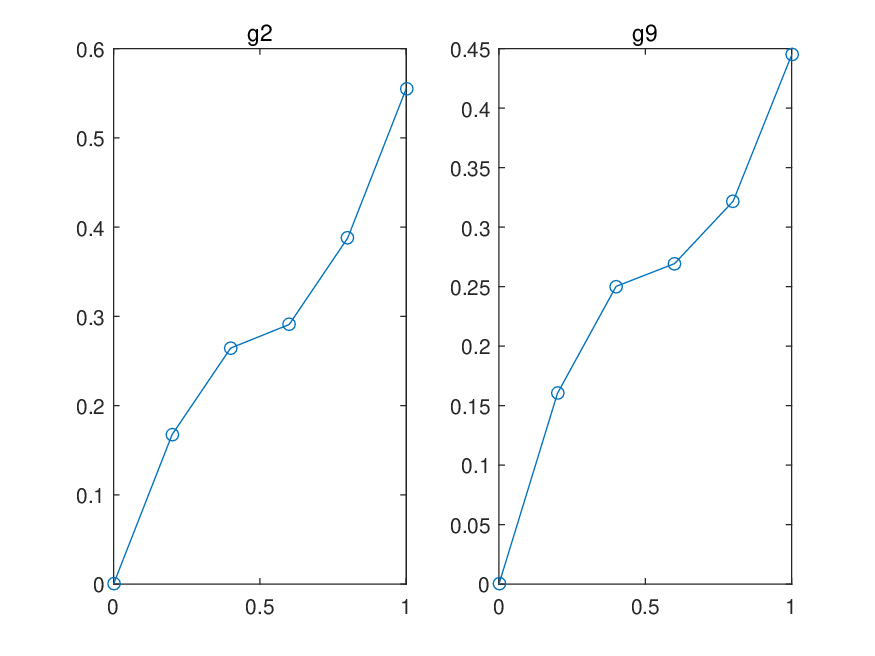}
  \caption{The marginal value functions for $p = 10$}
  \label{p10}
\end{figure}
\begin{figure}[!ht]
  \centering
  \includegraphics[scale=0.3]{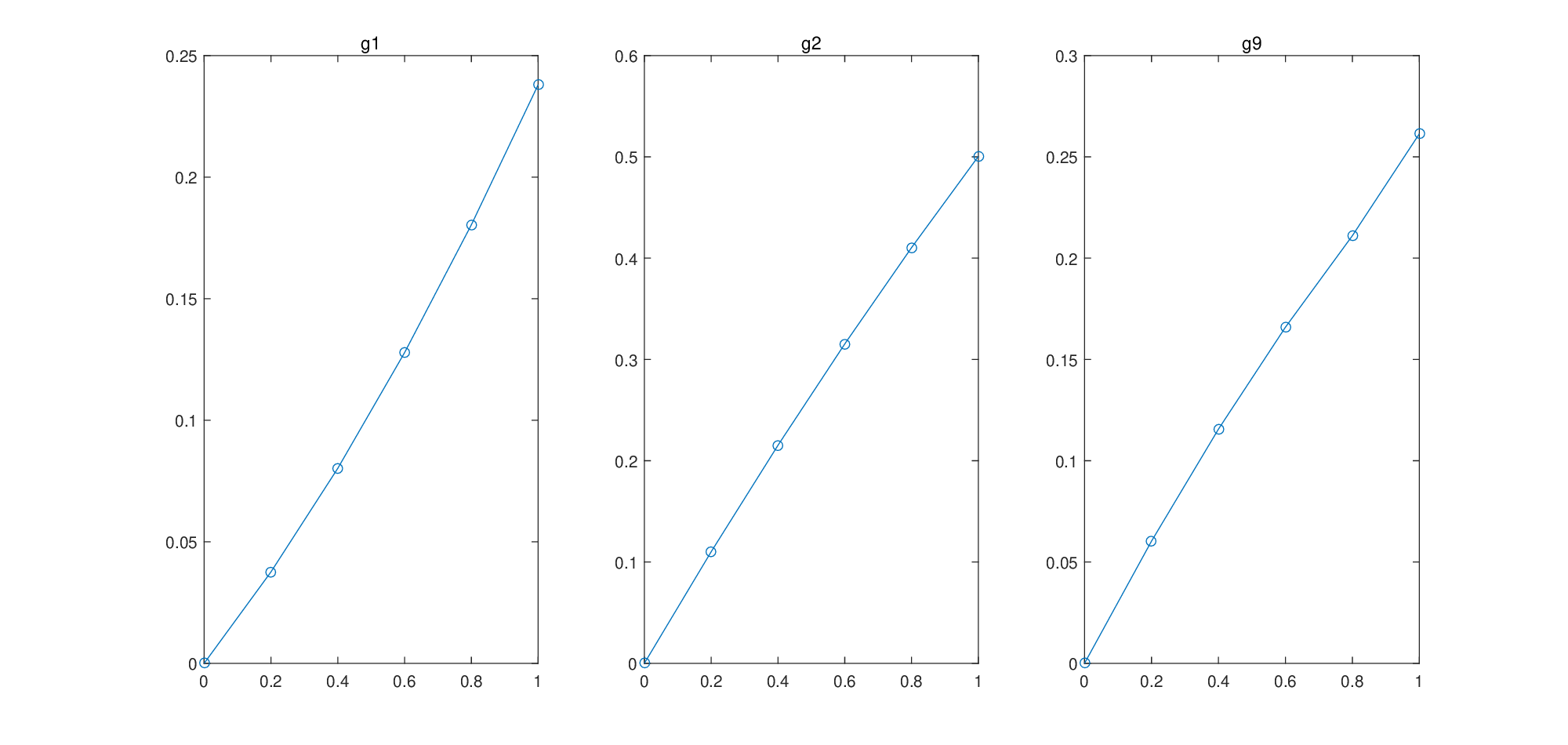}
  \caption{The marginal value functions for $p = 100$}
  \label{p100}
\end{figure}
\begin{figure}[!ht]
  \centering
  \includegraphics[scale=0.35]{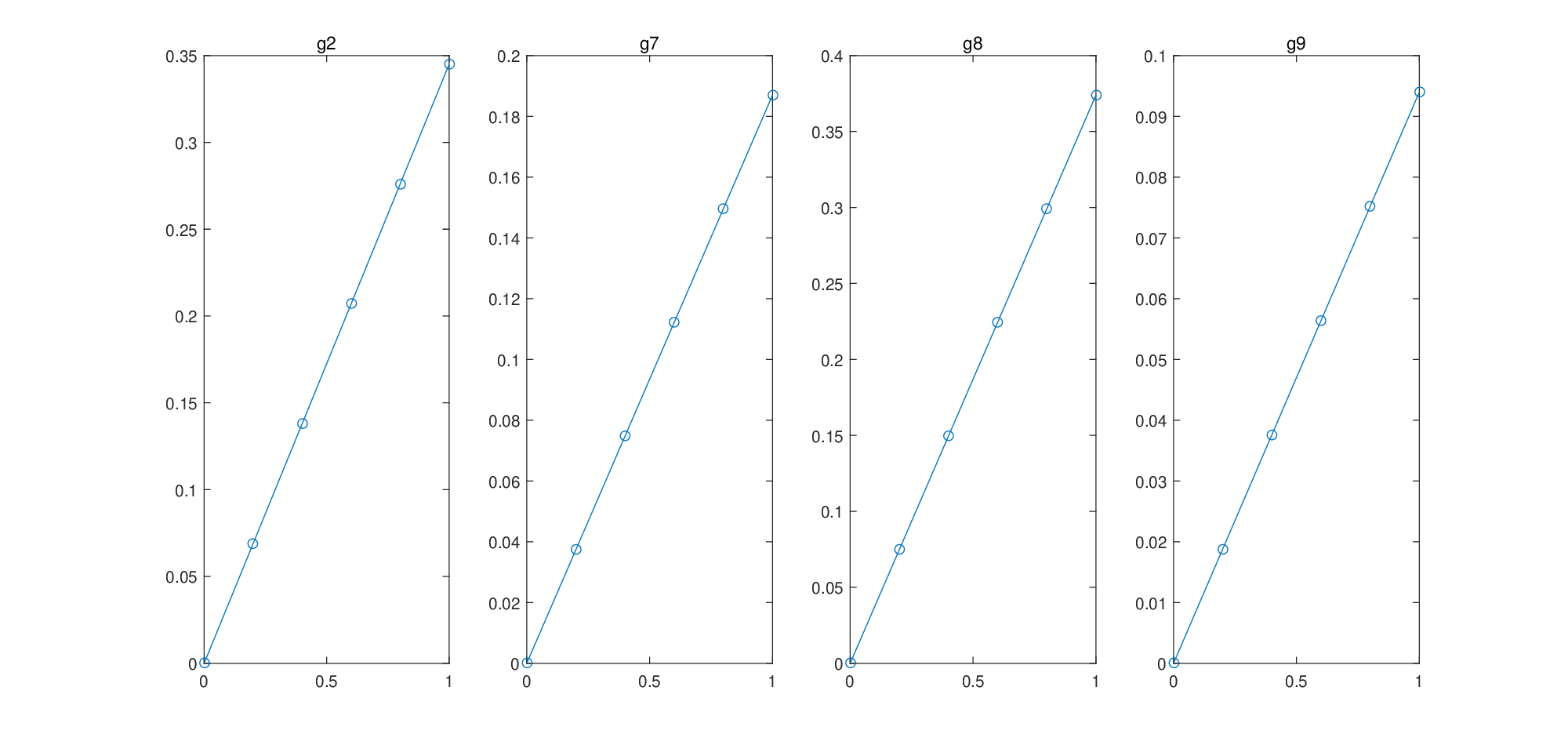}
  \caption{The marginal value functions for $p = 200$}
  \label{p200}
\end{figure}

\begin{table}[!ht]
\centering
\caption{Comprehensive scores of suppliers corresponding to different values of $p$}
\scalebox{0.8}{\begin{tabular}{c|cccccccccc}\toprule
   $p$&	$a_{1}$ & $a_{2}$ & $a_{3}$ & $a_{4}$ & $a_{5}$ & $a_{6}$ & $a_{7}$ & $a_{8}$ & $a_{9}$ & $a_{10}$ \\ \midrule
 $10$ & $0.6074$ & $0.6950$ & $0.3580$ & $0.6302$ & $0.6360$ & $0.3407$ & $0.3507$ & $0.3317$ & $0.3407$ & $0.3307$\\ 
 $100$ & $0.6462$ & $0.7706$ & $0.4920$ & $0.5021$ & $0.7248$ & $0.3292$ & $0.4444$ & $0.3168$ & $0.4344$ & $0.4245$\\ 
 $200$ & $0.5887$ & $0.5824$ & $0.5586$ & $0.5684$ & $0.6577$ & $0.5312$ & $0.5586$ & $0.4845$ & $0.5477$ & $0.5374$\\
   \bottomrule
 \end{tabular}}\label{cridit risk}
\end{table}

\subsection{The relevance of criteria}
Using the method introduced in Section \ref{main-model-2}, all possible streamlined supporting criteria sets are identified, as shown in Table \ref{Streamlined support criteria set}. The largest streamlined supporting criteria sets contain four criteria, suggesting that any larger supporting criteria set necessarily includes unnecessary criteria. Additionally, there are no core or redundant criteria, and any criteria could potentially be relied on by the DM to evaluate these suppliers. The relevance of criteria is shown in Table \ref{The importance of criteria}. It can be seen that the pollution production ($g_{9}$) is likely to be considered by the DM when evaluating suppliers, while the green image ($g_{2}$) and staff environmental training ($g_{10}$) are less likely to be taken into account. When the DM evaluates the suppliers based on four criteria, pollution production ($g_{9}$) can be regarded as a core criterion, while green image ($g_{2}$), resource consumption ($g_{4}$) and staff environmental training ($g_{10}$) are redundant criteria. In addition, the DM is most likely to evaluate suppliers on the basis of environmentally friendly technology ($g_{3}$), green competencies ($g_{5}$), environment management ($g_{6}$) and pollution production ($g_{9}$).
\begin{table}[!ht]
\centering
\caption{Streamlined supporting criteria sets ($4$ breakpoints)}
\scalebox{0.9}{\begin{tabular}{cccc}\toprule
   $\sum\limits_{j=1}^{m}\delta_{j}=2$ &	$\sum\limits_{j=1}^{m}\delta_{j}=3$ &$\sum\limits_{j=1}^{m}\delta_{j}=4$\\ \midrule
$\{g_{2},g_{9}\}$&$\{g_{3},g_{4},g_{9}\}$& $\{g_{3},g_{5},g_{6},g_{9}\}$\\
$\{g_{2},g_{8}\}$&$\{g_{3},g_{4},g_{8}\}$&$\{g_{3},g_{6},g_{8},g_{9}\}$\\
$\{g_{2},g_{5}\}$&$\{g_{3},g_{4},g_{5}\}$&$\{g_{6},g_{7},g_{8},g_{9}\}$\\
&	 $\{g_{4},g_{7},g_{9}\}$ &$\{g_{1},g_{3},g_{5},g_{9}\}$\\
&  $\{g_{1},g_{9},g_{10}\}$&$\{g_{5},g_{6},g_{7},g_{9}\}$\\
& $\{g_{6},g_{9},g_{10}\}$&$\{g_{1},g_{3},g_{6},g_{9}\}$\\
& $\{g_{7},g_{9},g_{10}\}$&$\{g_{1},g_{6},g_{7},g_{9}\}$\\
   \bottomrule
 \end{tabular}}\label{Streamlined support criteria set}
\end{table}
\begin{table}[!ht]
\centering
\caption{The relevance of criteria ($4$ breakpoints)}
\scalebox{0.9}{\begin{tabular}{ccccccccccc}\toprule
   &	$g_{1}$ & $g_{2}$ & $g_{3}$ & $g_{4}$ & $g_{5}$ & $g_{6}$ & $g_{7}$ & $g_{8}$ & $g_{9}$ & $g_{10}$ \\ \midrule
 $R(c_{j})$ & $4$ & $3$ & $7$ & $4$ & $5$ & $7$ & $5$ & $4$ & $13$ & $3$\\
   \bottomrule
 \end{tabular}}\label{The importance of criteria}
\end{table}

The above results are based on the assumption that the evaluation scale of the criteria is equally divided into 5 sub-intervals ($4$ breakpoints). Furthermore, we divide the evaluation scale of the criteria into 4 ($3$ breakpoints) and 6 ($5$ breakpoints) sub-intervals respectively, and redetermine the streamlined supporting criteria set for preference information and the relevance of criteria. Table (\ref{Streamlined support criteria set-4}) $-$ (\ref{The importance of criteria-6}) present the computational results.
\begin{table}[!ht]
\centering
\caption{Streamlined supporting criteria sets ($3$ breakpoints)}
\scalebox{0.9}{\begin{tabular}{cccc}\toprule
   $\sum\limits_{j=1}^{m}\delta_{j}=2$ &	$\sum\limits_{j=1}^{m}\delta_{j}=3$ &$\sum\limits_{j=1}^{m}\delta_{j}=4$\\ \midrule
$\{g_{2},g_{9}\}$&$\{g_{2},g_{5},g_{7}\}$& $\{g_{1},g_{4},g_{9},g_{10}\}$\\
$\{g_{2},g_{8}\}$&$\{g_{2},g_{3},g_{5}\}$&$\{g_{3},g_{6},g_{8},g_{9}\}$\\
&$\{g_{3},g_{4},g_{9}\}$&$\{g_{3},g_{5},g_{6},g_{9}\}$\\
&$\{g_{3},g_{4},g_{5}\}$ &$\{g_{1},g_{3},g_{9},g_{10}\}$\\
&  $\{g_{4},g_{7},g_{9}\}$&$\{g_{1},g_{7},g_{9},g_{10}\}$\\
& $\{g_{3},g_{4},g_{8}\}$&\\
& $\{g_{6},g_{9},g_{10}\}$&\\
   \bottomrule
 \end{tabular}}\label{Streamlined support criteria set-4}
\end{table}
\begin{table}[!ht]
\centering
\caption{The relevance of criteria ($3$ breakpoints)}
\scalebox{0.9}{\begin{tabular}{ccccccccccc}\toprule
   &	$g_{1}$ & $g_{2}$ & $g_{3}$ & $g_{4}$ & $g_{5}$ & $g_{6}$ & $g_{7}$ & $g_{8}$ & $g_{9}$ & $g_{10}$ \\ \midrule
 $R(c_{j})$ & $3$ & $4$ & $7$ & $5$ & $4$ & $3$ & $3$ & $3$ & $9$ & $4$\\
   \bottomrule
 \end{tabular}}\label{The importance of criteria-4}
\end{table}
\begin{table}[!ht]
\centering
\caption{Streamlined supporting criteria sets ($5$ breakpoints)}
\scalebox{0.9}{\begin{tabular}{cccc}\toprule
   $\sum\limits_{j=1}^{m}\delta_{j}=2$ &	$\sum\limits_{j=1}^{m}\delta_{j}=3$ &$\sum\limits_{j=1}^{m}\delta_{j}=4$\\ \midrule
$\{g_{2},g_{9}\}$ & $\{g_{3},g_{4},g_{9}\}$ & $\{g_{1},g_{4},g_{9},g_{10}\}$\\
$\{g_{2},g_{8}\}$ & $\{g_{3},g_{4},g_{5}\}$ & $\{g_{3},g_{5},g_{6},g_{9}\}$\\
$\{g_{2},g_{5}\}$ & $\{g_{4},g_{7},g_{9}\}$ & $\{g_{3},g_{6},g_{8},g_{9}\}$\\
                  & $\{g_{3},g_{4},g_{8}\}$  & $\{g_{5},g_{6},g_{7},g_{9}\}$ \\
                  & $\{g_{6},g_{9},g_{10}\}$& $\{g_{6},g_{7},g_{8},g_{9}\}$\\
                  &                         & $\{g_{1},g_{3},g_{9},g_{10}\}$\\
                  &                         & $\{g_{1},g_{3},g_{5},g_{9}\}$\\
                  &                         & $\{g_{1},g_{3},g_{8},g_{9}\}$\\
                  &                         & $\{g_{1},g_{3},g_{6},g_{9}\}$\\
                  &                         & $\{g_{1},g_{6},g_{7},g_{9}\}$\\
   \bottomrule
 \end{tabular}}\label{Streamlined support criteria set-6}
\end{table}
\begin{table}[!ht]
\centering
\caption{The relevance of criteria ($5$ breakpoints)}
\scalebox{0.9}{\begin{tabular}{ccccccccccc}\toprule
   &	$g_{1}$ & $g_{2}$ & $g_{3}$ & $g_{4}$ & $g_{5}$ & $g_{6}$ & $g_{7}$ & $g_{8}$ & $g_{9}$ & $g_{10}$ \\ \midrule
 $R(c_{j})$ & $6$ & $3$ & $9$ & $5$ & $5$ & $7$ & $4$ & $5$ & $14$ & $3$\\
   \bottomrule
 \end{tabular}}\label{The importance of criteria-6}
\end{table}

By comparing the streamlined supporting criteria sets and relevance of criteria determined by the proposed method under different assumptions regarding the number of breakpoints, we can find that:
\begin{itemize}
  \item The smallest streamlined supporting criteria set contains $2$ criteria and the largest one contains $4$ criteria.
  \item The collections of streamlined supporting criteria sets identified by the proposed method differ, yet there exist numerous common streamlined supporting criteria sets among them.
  \item The criteria most likely to be relied upon by the DM in evaluating alternatives is pollution production ($g_{9}$). In the case where the DM evaluates alternatives based on four criteria, pollution production ($g_{9}$) is the core criterion.
  \item The sum of relevance of all criteria in $\{g_{3},g_{5},g_{6},g_{9}\}$ and  $\{g_{1},g_{3},g_{9},g_{10}\}$ is maximized in the case of $3$ breakpoints, $\{g_{3},g_{5},g_{6},g_{9}\}$ in the case of $4$ breakpoints, and $\{g_{1},g_{3},g_{6},g_{9}\}$ in the case of 5 breakpoints. These streamlined supporting criteria sets are most likely to be relied upon by the DM when evaluating alternative options, and they are highly similar.   
\end{itemize}
In summary, the method proposed in this paper is relatively unaffected by the number of breakpoints when determining the streamlined supporting criteria sets for preference information and relevance of criteria. These streamlined supporting criteria sets and the relevance of criteria reveal some information about the DM's preference system.

\section{Conclusions}\label{conclusion}
In this paper, we developed a novel embedded criteria selection model using preference disaggregation technique and regularization theory. The method aims to infer the criteria used by DMs in evaluating reference alternatives, as well as the value function for aggregated criteria, based solely on a limited amount of decision examples. We introduced 0-1 variables to indicate whether a criterion is selected or not, and reformulated the additive value function to meet the needs of criteria selection. The fitting ability (empirical error) and complexity (generalization error) of value functions are used as quality measures of criteria subsets to guide criteria selection. The interaction between the construction of value functions and criteria selection ensures that the learned value functions possess good performance. Unlike the existing literature, which only considers the degree of deviation of marginal value functions from linearity as the complexity of value functions, we argue that the complexity of a value function is also related to the number of marginal value functions it contains.

The constructed criteria selection model uses the trade-off between empirical error and generalization error as the objective function to identify the criteria that the DM relies on when evaluating alternatives as well as the corresponding value functions. Such value functions have the simplest form while guaranteeing the fitting ability to the preference information. If there exists a value function constructed over a criteria subset with an empirical error of 0, the criteria subset is sufficiently discriminating and is called the supporting criteria set. To find all supporting criteria sets that do not contain any unnecessary criteria (called streamlined supporting criteria set), we employ an iterative approach. The more streamlined supporting criteria sets that contain a criterion, the more likely that the criterion is considered by the DM. Notably, criteria that are included in all streamlined supporting criteria sets are inevitably used by the DM in evaluating alternatives, while criteria that are not included in any streamlined supporting criteria set are definitely not considered.

Such embedded criteria selection method is intuitive and has good interpretability. The use of decision examples as inputs weakens the cognitive effort of the DM. This method enables a constructive interaction with the DM who can explore the learned value function and the selected criteria by modifying or enriching the preference information. This makes the proposed criteria selection method user-friendly. 

\subsection{Computational complexity analysis}
The criteria selection method proposed in this paper is based on solving mixed linear integer programming problems. There are many mature algorithms for such programming problems, and the computational complexity of these algorithms is mainly determined by the number of variables and constraints. Therefore, we primarily analyze the relationship between the number of variables and constraints and the scale of decision problems. Assuming a MCDM problem involves a total of $m$ criteria, the evaluation scales for these criteria are equally divided into $\gamma$ sub-intervals, and the DM has provided $l$ pieces of preference information beforehand. Examining model \ref{regularization-1} (model \ref{regularization-2}) constructed for such a decision problem, we find that it has a total of $m$ 0-1 variables and $m\times\gamma$ continuous variables. Furthermore, besides the monotonicity and normalization conditions, $[SC]$ includes $m\times(\gamma-1)$ constraints, and $E^{A^{R}}$ contains $l$ constraints. After performing a mathematical transformation, model \ref{regularization-1} (model \ref{regularization-2}) is transformed into model \ref{regularization-1*} (model \ref{regularization-2*}), with an additional $m$ variables and $4m\times\gamma$ linear constraints introduced. Therefore, we ultimately need to solve a mixed-integer linear programming problem involving $(2+\gamma)\times m$ (including $m$ 0-1 variables) variables and $(5\gamma-1)\times m+l$ linear constraints (excluding monotonicity conditions, normalization conditions, and constraints indicating the sign of variables). At present, mainstream optimization solvers such as Gurobi, COPT can solve mixed integer programming problems with thousands of variables and constraints in a reasonable amount of time. The actual solving time is related to the performance of computers, the setting of parameters, and the choice of solving strategy. Moreover, if appropriate algorithms and solving strategies are designed based on the characteristics of programming problems, the computation time can be significantly reduced.
\subsection{Research prospects}
In the future, further exploration can be conducted on how the interaction between criteria affects the selection of criteria. In such scenario, the aggregation model should be constructed to satisfy the needs of criteria selection and capture the interactions between criteria. Attention should also be paid to how the interactions between criteria affect the quality of criteria subsets. Moreover, the embedded criteria selection method can be applied within robust ordinal regression where all instances of compatible value functions are used to build necessary and possible preference relations in the set of alternatives. At the same time, the stability of the supporting criteria set (streamlined supporting criteria set) is also a topic worthy of further research.

\section*{CRediT authorship contribution statement}
\textbf{Kun Zhou}: Conceptualization, Methodology, Writing - Original Draft, Writing - Review and Editing, Data Curation, Visualization;
\textbf{Zaiwu Gong}: Conceptualization, Supervision, Funding acquisition, Validation, Formal analysis;
\textbf{Guo Wei}: Writing - Original Draft, Writing - Review \& Editing, Validation, Formal analysis;
\textbf{Roman Słowi\'{n}ski}: Writing - Review \& Editing, Conceptualization, Methodology, Validation.

\section*{Declaration of competing interest}
The authors declare that they have no known competing financial interests or personal relationships that could have appeared to
influence the work reported in this paper.

\section*{Data availability}
No data was used for the research described in the article.

\section*{Acknowledgements}
This work was supported by the National Natural Science Foundation of China (Grant numbers 72371137). Roman S{\l}owi\'{n}ski wishes to acknowledge the support of the National Research Council (NCN) for the MAESTRO project no. 2023/50/A/HS4/00499 on "Intelligent decision support based on explanatory analytics of preference data".






\appendix
\section{The proof of Proposition \ref{gg}}\label{A1}
\textbf{Proposition} \ref{gg} Given two criteria subsets, $\mathcal{C}_{k}$ and $\mathcal{C}_{t}$, satisfying $\mathcal{C}_{k}\neq\emptyset$ and $\mathcal{C}_{k}\subseteq\mathcal{C}_{t}$. If $\mathcal{C}_{k}$ is a supporting criteria set for preference information, $\mathcal{C}_{t}$ is also a supporting criteria set.
\begin{proof}
If $\mathcal{C}_{k}=\mathcal{C}_{t}$, the proposition is obviously true. Otherwise, $\mathcal{C}_{k}$ is a proper subset of $\mathcal{C}_{t}$. According to Definition \ref{supporting}, there exists a $\mathcal{C}_{k}$-based value function $u(\cdot)=\sum\limits_{\{j|c_{j}\in \mathcal{C}_{k}\}}\textbf{V}_{j}^{\textrm{T}}\times \textbf{A}_{j}(\cdot)$ that can restore the preference information. A $\mathcal{C}_{t}$-based value function $u^{\star}(\cdot)=\sum\limits_{\{j|c_{j}\in \mathcal{C}_{t}\}}(\textbf{V}^{\star}_{j})^{\textrm{T}}\times \textbf{A}_{j}(\cdot)$ is constructed, satisfying the following two conditions:
\begin{enumerate}
  \item[(1)] For $p\in \{j|c_{j}\in \mathcal{C}_{k}\}$, $\textbf{V}^{\star}_{p}=\textbf{V}_{p}$;
  \item[(2)] For $p\in \{j|c_{j}\in \mathcal{C}_{t}\setminus\mathcal{C}_{k}\}$, $\textbf{V}^{\star}_{p}=\textbf{0}$.
\end{enumerate}
It is easy to see that $u^{\star}(\cdot)$ satisfies the monotonicity and normalization conditions. For any $a_{i}\in \mathcal{A}$, we have $u^{\star}(a_{i})=u(a_{i})$. Since $u(\cdot)$ can restore the DM's preference information, $u^{\star}(\cdot)$ can also achieve this. This shows that $\mathcal{C}_{t}$ is a supporting criteria set.
\end{proof}

\section{The proof of Proposition \ref{main-proof}}\label{A2}
\textbf{Proposition} \ref{main-proof} Model (\ref{regularization-1}) and model (\ref{regularization-1*}) are equivalent.
\begin{proof}
To prove the equivalence of these optimization models, it is sufficient to show that given a feasible solution to one model, a corresponding feasible solution to the other model can be found, and the respective objective function values are identical. The proof consists of two parts: \\ a) Assume that $\overline{\varphi},\ \ \{\overline{\textbf{V}_{j}^{\textrm{T}}},\ j\in M\},\ \{\overline{\delta_{j}},\ j\in M\},\ \{\overline{\sigma_{+}(a_{i})}, \ \overline{\sigma_{-}(a_{i})},\ a_{i}\in A^{R}\}$ is a feasible solution to model (\ref{regularization-1}) with the objective function value of $\sum\limits_{a_{i}\in A^{R}}\left(\overline{\sigma_{+}(a_{i})}+\overline{\sigma_{-}(a_{i})}\right)+C\times\left(m\times\overline{\varphi}+\sum\limits_{j=1}^{m}\overline{\delta_{j}}
\right)$. Given $j\in M$, for any $s=1,2,\cdots,\lambda_{j}$, if $\delta_{j}=0$, $\overline{z_{j}^{s}}=0$, otherwise $\overline{z_{j}^{s}}=\overline{\triangle u_{j}^{s}}$. It is evident that $\{\overline{\textbf{Z}_{j}^{\textrm{T}}},\ j\in M\}$, $\{\overline{\delta_{j}},\ j\in M\}$ and $\{\overline{\textbf{V}_{j}^{\textrm{T}}},\ j\in M\}$ satisfy constraints $[LC]$. Furthermore, since
$\overline{\textbf{Z}_{j}^{\textrm{T}}}\times \textbf{A}_{j}(a_{i})=\delta_{j}\times \overline{\textbf{V}_{j}^{\textrm{T}}}\times \textbf{A}_{j}(a_{i})$, it is straightforward to verify that solution $\overline{\varphi},\ \{\overline{\textbf{V}_{j}^{\textrm{T}}},\ j\in M\},\{\overline{\delta_{j}},\ j\in M\},\{\overline{\sigma_{+}(a_{i})},\overline{\sigma_{-}(a_{i})},\ a_{i}\in A^{R}\},\{\overline{\textbf{Z}^{\textrm{T}}_{j}},\ j\in M\}$ satisfies the constraints $[MC]$,$[NC]$,$[SC]$ and $E^{A^{R}}$ of model (\ref{regularization-1*}). Therefore, it is a feasible solution to model (\ref{regularization-1*}), and the corresponding objective function value is also $\sum\limits_{a_{i}\in A^{R}}\left(\overline{\sigma_{+}(a_{i})}+\overline{\sigma_{-}(a_{i})}\right)+C\left(m\times\overline{\varphi}+\sum\limits_{j=1}^{m}\overline{\delta_{j}}
\right)$.\\
b) Conversely, assume that $\overline{\varphi},\ \{\overline{\textbf{V}_{j}^{\textrm{T}}},\ j\in M\},\{\overline{\delta_{j}},\ j\in M\},\{\overline{\sigma_{+}(a_{i})},\overline{\sigma_{-}(a_{i})},\ a_{i}\in A^{R}\},\{\overline{\textbf{Z}^{\textrm{T}}_{j}}, j\in M\}$ is a feasible solution to model (\ref{regularization-1*}). According to constraints $[LC]$, we have $\overline{\textbf{Z}_{j}^{\textrm{T}}}\times \textbf{A}_{j}(a_{i})=\delta_{j}\times \overline{\textbf{V}_{j}^{\textrm{T}}}\times \textbf{A}_{j}(a_{i})$. It is easy to verify that $\overline{\varphi},\ \{\overline{\textbf{V}_{j}^{\textrm{T}}},\ j\in M\},\ \{\overline{\delta_{j}},\ j\in M\},\ \{\overline{\sigma_{+}(a_{i})},\ \overline{\sigma_{-}(a_{i})},\ a_{i}\in A^{R}\}$ satisfy all constraints of model (\ref{regularization-1}), and the corresponding objective function values for both models are $\sum\limits_{a_{i}\in A^{R}}\left(\overline{\sigma_{+}(a_{i})}+\overline{\sigma_{-}(a_{i})}\right)+C\times\left(m\times\overline{\varphi}+\sum\limits_{j=1}^{m}\overline{\delta_{j}}
\right)$.\\
In summary, model (\ref{regularization-1}) and model (\ref{regularization-1*}) are equivalent.
\end{proof}

\section{The proof of Proposition \ref{well-define}}\label{A3}
\textbf{Proposition} \ref{well-define} The value function determined by model (\ref{regularization-1*}) is non-degenerate.
\begin{proof}
Assume that $\overline{\varphi},\ \{\overline{\textbf{V}_{j}^{\textrm{T}}},\ j\in M\},\ \{\overline{\delta_{j}},\ j\in M\},\ \{\overline{\sigma_{+}(a_{i})},\ \overline{\sigma_{-}(a_{i})},\ a_{i}\in A^{R}\}$, $\{\overline{\textbf{Z}^{\textrm{T}}_{j}}, j\in M\}$ is the optimal solution to model (\ref{regularization-1*}) with the optimal value $O^{\ast}$. Then, $\{c_{j}|\overline{\delta_{j}}=1\}$ contains the selected criteria, and the corresponding value function is $u(\cdot)=\sum\limits_{\{c_{j}|\overline{\delta_{j}}=1\}}\overline{\textbf{V}_{j}^{\textrm{T}}}\times \textbf{A}_{j}(\cdot)$. If $u(\cdot)$ is degenerate, there exists $c_{k}\in \{c_{j}|\overline{\delta_{j}}=1\}$ such that $\overline{\textbf{V}_{k}^{\textrm{T}}}=\textbf{0}$. Let $\overline{\delta_{k}}=0$ and keep the other decision variables unchanged. Then it is easy to verify that this new instance of decision variables satisfies all constraints of model (\ref{regularization-1*}),  and the corresponding objective function value is $O^{\ast}-1<O^{\ast}$. This implies that $\overline{\varphi},\ \{\overline{\textbf{V}_{j}^{\textrm{T}}},\ j\in M\},\ \{\overline{\delta_{j}},\ j\in M\},\ \{\overline{\sigma_{+}(a_{i})},\overline{\sigma_{-}(a_{i})},\ a_{i}\in A^{R}\},\ \{\overline{\textbf{Z}^{\textrm{T}}_{j}},\ j\in M\}$ is not the optimal solution to model (\ref{regularization-1*}), contradicting the initial assumption. Therefore, $u(\cdot)=\sum\limits_{\{j|\overline{\delta_{j}}=1\}}\overline{\textbf{V}_{j}^{\textrm{T}}}\times \textbf{A}_{j}(\cdot)$ is non-degenerate.
\end{proof}

\section{The proof of Proposition \ref{reduction-set}}\label{A4}
\textbf{Proposition} \ref{reduction-set} The supporting criteria sets determined by the above method are all streamlined.
\begin{proof}
Recall that in this method, whenever a supporting criteria set $\mathcal{C}_{k}$ is identified, the constraint $\sum\limits_{\{j|c_{j}\in \mathcal{C}_{k}\}}\delta_{j}\leq|\mathcal{C}_{k}|-1$ is added, and model (\ref{regularization-2*}) is re-solved to find the next supporting criteria set. Due to the existence of constraint $\sum\limits_{\{j|c_{j}\in \mathcal{C}_{k}\}}\delta_{j}\leq|\mathcal{C}_{k}|-1$, newly identified supporting criteria set cannot contain $\mathcal{C}_{k}$. Therefore, the supporting criteria sets determined by this method are all streamlined.
\end{proof}

\end{document}